\documentclass[11pt]{amsart}
\usepackage{latexsym, amssymb, url, mathrsfs, amsfonts, MnSymbol}
\usepackage[all,cmtip]{xy}
\usepackage[pdftex,lmargin=1.45in, rmargin=1.45in,tmargin=1.25in,bmargin=1.25in, marginpar=3.25cm]{geometry}

\usepackage{color}

\newcommand{\ci}{C^{\infty}}
\newcommand{\IC}{\mathbb{C}}

\newcommand{\IQ}{\mathbb{Q}}
\newcommand{\uA}{\underline{A}}
\newcommand{\GL}{\mathrm{GL}}

\newcommand{\ZZ}{\mathbb{Z}}
\newcommand{\isomto}{\overset{\sim}{\rightarrow}}
\newcommand{\CC}{\mathcal{C}}

\newcommand{\hdr}{H_{\mathrm{dR}}}

\newcommand{\uo}{\underline{\omega}}
\newcommand{\Ig}{\mathrm{Ig}}
\newcommand{\ord}{\mathrm{ord}}
\newcommand{\muord}{\mu\mbox{-}\ord}

\newcommand{\CO}{\mathcal{O}}
\newcommand{\hdrone}{H^1_{\mathrm{dR}}}
\newcommand{\Fil}{\mathrm{Fil}}
\newcommand{\Image}{\mathrm{Image}}
\newcommand{\Gr}{\mathrm{Gr}}
\newcommand{\dual}{^\vee}
\newcommand{\KS}{\mathrm{KS}}
\newcommand{\ks}{\mathrm{ks}}
\newcommand{\id}{\mathrm{id}}
\newcommand{\cmfield}{F}%the CM field acting on the abelian varieties in the PEL moduli problem

\newcommand{\Ocmfield}{\CO_\cmfield}
\newcommand{\Igmu}{\Ig_\mu}%mu-ordinary Igusa tower
\newcommand{\Igmunm}{(\Igmu)_{n,m}}
\newcommand{\CA}{\mathcal{A}}
\newcommand{\CE}{\mathcal{E}}
\newcommand{\CS}{\mathcal{S}}
\newcommand{\CZ}{\mathcal{Z}}
\newcommand{\Auniv}{\CA_\mu}%universal abelian scheme over the mu-ordinary locus
\newcommand{\Smu}{\CS_\mu}%mu-ordinary locus
\newcommand{\schur}{\mathbb{S}}

\newcommand{\Aoverig}{\left(\Auniv/\Igmu\right)}
\newcommand{\st}{\mathrm{St}}%denotes standard representation
\newcommand{\Vmu}{V^N}%space of p-adic automorphic forms over the mu-ordinary locus \Igmu
\newcommand{\cf}{\mathfrak{f}}

\newcommand{\CL}{\mathcal{L}}

\newcommand{\Spec}{\mathrm{Spec}}
\newcommand{\gr}{\mathrm{gr}}

\newcommand{\rk}{\mathrm{rk}}%shorthand for ``rank''
\newcommand{\Emu}{\CE_\mu}
\newcommand{\Isom}{\mathrm{Isom}}
\newcommand{\can}{\mathrm{can}}
\newcommand{\ellcan}{\ell}

\newcommand{\End}{\mathrm{End}}
\newcommand{\IR}{\mathbb{R}}
\newcommand{\op}{\mathrm{op}}
\newcommand{\tildeellcan}{\tilde{\ell}_{\can}}
\newcommand{\Hom}{\mathrm{Hom}}
\newcommand{\Gm}{\mathbb{G}_m}
\newcommand{\diag}{\mathrm{diag}}
\newcommand{\unitrootold}{U}

\newcommand{\Para}{P}
\newcommand{\para}{\Para_\mu}
\newcommand{\Levi}{J}
\newcommand{\Levii}{J'}
\newcommand{\levi}{\Levi_\mu}

\newcommand{\uni}{U_\mu}
\newcommand{\Uni}{U}
\newcommand{\Borel}{B}
\newcommand{\Borell}{B'}
\newcommand{\borel}{\Borel_\mu}
\newcommand{\Nilp}{N}
\newcommand{\Nilpp}{N'}
\newcommand{\nilp}{\Nilp_\mu}

\newcommand{\bD}{{\mathbb D}}
\newcommand{\Q}{{\mathbb Q}}
\newcommand{\W}{{\mathbb W}}
\newcommand{\Z}{{\mathbb Z}}
\newcommand{\BK}{{\mathbb K}}
\newcommand{\F}{{\mathbb F}}

\newcommand{\cR}{\mathcal{R}}
\newcommand{\bG}{\mathbb{G}}
\newcommand{\uu}{{\underline{u}}}

\newcommand{\M}{{\mathfrak M}}
\newcommand{\FO}{{\mathfrak O}}
\newcommand{\Pmu}{{{\mathcal P}_\mu}}

\newcommand{\Torus}{T}
\newcommand{\Toruss}{T'}
\newcommand{\torus}{T_\mu}

\newcommand{\C}{{\mathbb C}}
\newcommand{\A}{{\mathbb A}}
\newcommand{\X}{{\mathbb X}}

\newcommand{\Sh}{\mathit{Sh}}
\newcommand{\sh}{\mathit{sh}}
\newcommand{\K}{\mathcal{K}}
\newcommand{\CD}{\mathcal{D}}
\newcommand{\co}{\mathfrak{o}}
\newcommand{\T}{\mathcal{T}}
\newcommand{\G}{\mathcal{G}}
\newcommand{\ck}{\kappa(\p)}
\newcommand{\p}{\mathfrak{p}}

\newcommand{\spec}{\mathrm{Spec}}
\newcommand{\cS}{\mathcal{S}}
\newcommand{\bS}{\overline{\mathcal{S}}}
\newcommand{\hasse}{{E_\mu}}

\newcommand{\BT}{Barsotti--Tate }

\newcommand{\ud}{\underline{D}}
\newcommand{\unipo}{{\mathbb U}}
\newcommand{\U}{\underline{U}}

\newcommand{\Dring}{\mathscr{D}}

\newcommand{\reflexfield}{E}

\renewcommand{\MR}[1]{ }

\theoremstyle{plain}

\newtheorem{thm}{Theorem}
\numberwithin{thm}{subsection}
\newtheorem{cor}[thm]{Corollary}

\newtheorem{prop}[thm]{Proposition}

\theoremstyle{definition}
\newtheorem{definition}[thm]{Definition}
\newtheorem{defi}[thm]{Definition}
\newtheorem{defn}[thm]{Definition}
 
\theoremstyle{remark}
\newtheorem{remark}[thm]{Remark}
\newtheorem{rmk}[thm]{Remark}

\title[$p$-adic families in the $\mu$-ordinary setting]{$p$-adic families of automorphic forms in the $\mu$-ordinary setting}

\date{\today}

\author[E.\  Eischen]{E.\ Eischen$^*$}
\thanks{$^*$Partially supported by NSF Grant DMS-1559609 and NSF CAREER Grant DMS-1751281.}
\author[E.\ Mantovan]{E.\ Mantovan}

\address{E. E. Eischen\\
Department of Mathematics\\
University of Oregon\\
Fenton Hall\\
Eugene, OR 97403-1222\\
USA}
\email{eeischen@uoregon.edu}

\address{E. Mantovan\\
 Department of Mathematics\\
Caltech\\
Pasadena, CA 91125\\
 USA}
 \email{mantovan@caltech.edu}

\begin{document}

\bibliographystyle{amsalpha}  

\maketitle
\vspace{-0.25in}
\begin{abstract}
We develop a theory of $p$-adic automorphic forms on unitary groups that allows $p$-adic interpolation in families and holds for all primes $p$ that do not ramify in the reflex field $E$ of the associated unitary Shimura variety. 
If the ordinary locus is nonempty (a condition only met if $p$ splits completely in $E$), we recover Hida's theory of $p$-adic automorphic forms, which is defined over the ordinary locus.  More generally, we work over the $\mu$-ordinary locus, which is open and dense.  

By eliminating the splitting condition on $p$, our framework should allow many results employing Hida's theory to extend to infinitely many more primes.
We also provide a construction of $p$-adic families of automorphic forms that uses differential operators constructed in the paper.  Our approach is to adapt the methods of Hida and Katz to the more general $\mu$-ordinary setting, while also building on papers of each author.  Along the way, we encounter some unexpected challenges and subtleties that do not arise in the ordinary setting.   
 \end{abstract}

\setcounter{tocdepth}{1}

\section{Introduction}\label{introduction-section}
The $p$-adic theory of modular forms plays a powerful role in number theory.  Its reach includes the proof of Fermat's Last Theorem, proofs of instances of the main conjecture of Iwasawa theory, a realization of the Witten genus in homotopy theory, and constructions of $p$-adic $L$-functions.   Geometric developments continue to expand the impact of the $p$-adic theory, for example in settings employing automorphic forms on unitary groups.

Shortly after J.-P.\ Serre defined $p$-adic modular forms as $p$-adic limits of Fourier expansions of classical modular forms, N.\ Katz gave a geometric reformulation \cite{serre, ka2}.  H.\ Hida later extended Katz's geometric framework to $p$-adically interpolate automorphic forms on many reductive groups, including unitary groups \cite{Hida, hi05}.  This geometric approach realizes $p$-adic automorphic forms inside a vector bundle over (a cover of) the {\it ordinary locus} of a Shimura variety.

The present paper is the first in a projected multi-paper project to extend Hida's theory to the {\it $\mu$-ordinary locus} of each unitary Shimura variety $\mathcal{S}$, for all rational primes $p$ that do not ramify in the reflex field $E$ of $\mathcal{S}$.  When $p$ does not split completely in $E$, the ordinary locus is empty, in which case Hida's theory concerns functions on the empty set.  On the other hand, for $p$ that does not ramify in $E$, the $\mu$-ordinary locus is an open, dense stratum, which is the same as the ordinary locus when $p$ splits completely, as shown by T.\ Wedhorn in \cite{wedhorn}.  Our approach holds over the $\mu$-ordinary locus for $p$ unramified and specializes to Hida's theory when the ordinary locus is nonempty.

As Hida's theory has played a substantial role in various applications, it is natural to try to extend its impact still further by using the $\mu$-ordinary locus to remove the splitting condition on the prime $p$.   In turn, this should facilitate a link with Hida's $P$-ordinary automorphic forms, with $P$ a certain parabolic, introduced in \cite{HidaAsian}.  It should also be possible to adapt existing constructions of $p$-adic $L$-functions (e.g., \cite{SkUr, MLH, EW, EFMV, hsieh-jams, EHLS}) to the $\mu$-ordinary setting.

Proofs employing $p$-adic automorphic forms typically require not only a definition facilitating interpolation in families but also connections with the theory of $\IC$-valued automorphic forms.  For example, most applications of Hida's, Katz's, and Serre's theories employ Hecke operators or $p$-adic analogues of the Maass--Shimura differential operators.  Thus, beyond defining $p$-adic automorphic forms over (a cover of) the $\mu$-ordinary locus (an adaptation of \cite[Chapter 8]{Hida} made possible by geometric developments of Wedhorn, B.\ Moonen, W.\ Goldring, M.-H.\ Nicole, and the second named author in \cite{wedhorn, moonen, GN, mantovan1}), important goals include achieving $\mu$-ordinary analogues of the following results already established for Hida's setting:
\begin{enumerate}
\item{Realize classical automorphic forms in the space of $p$-adic automorphic forms.}\label{classical-item}
\item{Construct $p$-adic differential operators analogous to the $\ci$ Maass--Shimura operators, and explicitly describe their action on local expansions of $p$-adic automorphic forms at certain points (e.g., Serre--Tate expansions, analogues of $q$-expansions).
}\label{differential-item}
\item{Study the action of Hecke operators on $p$-adic automorphic forms and cut out a space of ($\mu$-)ordinary automorphic forms via (a $\mu$-ordinary analogue of) Hida's ordinary projector, in order to develop a theory of ($\mu$-)ordinary forms and families; and relate these families to systems of Hecke eigenvalues, as well as to certain holomorphic automorphic forms.}\label{action-item}
\item{Develop a notion of overconvergence in the $\mu$-ordinary setting, and construct eigenvarieties parametrizing overconvergent families.}\label{over-item}
\item{Determine relationships between the classical forms in Item \eqref{classical-item} and overconvergence from Item \eqref{over-item}, by studying the action of the Hecke operators from Item \eqref{action-item} and the action of the differential operators from Item \eqref{differential-item}.}\label{combo-item}
\end{enumerate}
Some constructions turn out to be more delicate and involved than one might expect from the ordinary case alone.  Consequently, this has become an active research area.

In the present paper, after defining $p$-adic automorphic forms over the $\mu$-ordinary locus (in analogue with Hida's construction over the ordinary locus), we accomplish Item \eqref{classical-item} under certain conditions on the weights of the classical forms, thus extending \cite[Chapter 8]{Hida} from the ordinary setting.  These conditions are forced by the underlying representation theory in the $\mu$-ordinary case, as seen in Section \ref{realization-section}.  We can obtain stronger results about embeddings locally than globally, which suffice for many applications.

The present paper also constructs the differential operators from Item \eqref{differential-item}, thus extending the constructions from \cite{E09, EDiffOps, kaCM, EFMV} to the $\mu$-ordinary setting.   We also describe the action of the differential operators on Serre--Tate expansions of $p$-adic automorphic forms, but only when the weights meet certain conditions (a condition we suspect can be considerably lessened with additional nontrivial, technical work concerning Lubin--Tate group laws).  Our explicit description enables us to establish congruences between $p$-adic automorphic forms over the $\mu$-ordinary locus and, as a consequence, construct $p$-adic families of automorphic forms (currently only under conditions on the weights), like in, e.g., \cite{kaCM, hida-differential, panch, CoPa, apptoSHL, apptoSHLvv, emeasurenondefinite, EFMV}.  E.\ de Shalit and E.\ Goren have studied similar differential operators for unitary groups of signature $(2,1)$ for quadratic imaginary fields \cite{dSG}.

In the sequel, we plan to build on the framework developed in the present paper to study aspects of \eqref{action-item} and \eqref{combo-item} in the case of families, including non-classical weights and the relationship with systems of Hecke eigenvalues.  One anticipated application will be a construction of new $p$-adic $L$-functions for unitary Shimura varieties with empty ordinary locus.  This work will also build on Hida's work on $P$-ordinary automorphic forms \cite{HidaAsian}.  

Some of \eqref{over-item} and \eqref{combo-item} has been achieved by S.\ Bijakowski, V.\ Pilloni, and B.\ Stroh in \cite{BPS, bijakowski} for classical weights (and not for families), via a different approach from the one in the present paper (and not involving Igusa towers, a necessary ingredient for interpolation in families in our project).  After proving the existence of canonical subgroups on a neighborhood of the $\mu$-ordinary locus of the associated Shimura variety and then adapting the analytic continuation methods of K.\ Buzzard and P.\ Kassaei \cite{buzz, kass1, kass2}, Bijakowski, Pilloni, and Stroh develop a notion of overconvergence naturally extending the one from the ordinary setting.  Then they prove overconvergent forms of small slope are classical, extending Coleman's and Hida's classicality results \cite{hida86, coleman1, coleman2} that followed F.\ Gouv\^ ea and B.\ Mazur's conjectures \cite{gouvea-mazur}.

In the few weeks after the first version of this paper appeared on the arXiv, three additional papers addressing some of the above goals have been completed.  In \cite{hernandez1}, V.\ Hernandez constructs eigenvarieties when the signature is $(2, 1)$, which he plans to extend to other signatures.  In \cite{brasca-rosso}, R.\ Brasca and G.\ Rosso extend Hida theory to $\Lambda$-adic cuspidal $\mu$-ordinary forms (which appear to be closely related to Hida's $P$-ordinary forms), for $\Lambda$ a twisted Iwasawa algebra.  Extending \cite{dSG}, de Shalit and Goren recently completed a paper constructing differential operators over a $\mu$-ordinary Igusa tower for unitary groups of arbitrary signature for quadratic imaginary extensions of $\IQ$ \cite{DSG2}.  While their approach differs from ours (and, unlike ours, does not employ a replacement for the unit root submodule that appears to play a key role in the constructions in \cite{kaCM, EFMV, EDiffOps}), they expect their differential operators to coincide with ours (at least for quadratic imaginary extension of $\IQ$, the setting in which they work, whereas we also consider arbitrary CM fields).  They also discuss analytic continuation beyond the $\mu$-ordinary strata, as well as the action on certain Fourier--Jacobi expansions.

\subsection{Main results and organization of the paper}\label{mr-section}

Section \ref{preliminaries-section} summarizes key information about unitary Shimura varieties, the $\mu$-ordinary locus,  automorphic forms, and representation theory.  The seemingly bland observations in Section \ref{littlewood-section} about the possibility that representations occurring in the restrictions of irreducible representations might have multiplicity greater than $1$ play a major role in Section \ref{congruences-section}.  The root of the issue is that the trivialization of the sheaf of differentials of each ordinary abelian variety that plays a key role in the ordinary setting is replaced in the $\mu$-ordinary setting by a trivialization of the associated graded module of the sheaf of differentials of each $\mu$-ordinary abelian variety (with the associated graded module coming from the slope filtration at each $\mu$-ordinary point), or equivalently, the action of a general linear group on the ordinary Igusa tower gets replaced by the action of its parabolic subgroup preserving the slope filtration.  For notational convenience, we exclude some instances of slope $1/2$ in this paper.  We expect no mathematical problems handling that slope, but it would involve still more notation, since it involves working with unitary groups that are not isomorphic to copies of general linear groups.

In Section \ref{tower-section}, we construct the {\it $\mu$-ordinary Igusa tower} as a profinite \'etale cover of the formal $\mu$-ordinary locus.  While Hida's ordinary Igusa tower employs the structure of the $p$-torsion of the universal abelian variety, our analogue in the $\mu$-ordinary case uses the structure of the {\it associated graded module coming from the slope filtration} on the $p$-torsion.  
Section \ref{tower-section} concludes with a study of local expansions at points of the $\mu$-ordinary Igusa tower, which is necessary for proving the congruences in Theorem \ref{intro-thm1} below.  

Section \ref{congruences-section} introduces $p$-adic automorphic forms over the $\mu$-ordinary locus and Igusa tower.  When the ordinary locus is nonempty, Hida's definitions and ours coincide.  The introduction of filtrations in the $\mu$-ordinary case complicates some aspects of the theory.  One of the key features of Hida's theory, which makes many applications possible, is the embedding of the spaces of classical automorphic forms into the line bundle of $p$-adic automorphic forms over the ordinary Igusa tower.  Instead, there is a realization of appropriate subspaces -- but, if the ordinary locus is empty, not the whole space -- of automorphic forms over the $\mu$-ordinary locus in the space $V$ of $p$-adic automorphic forms over the Igusa tower, as explained in Section \ref{realization-section}.  Section \ref{congruences-section} concludes with results about congruences:
\begin{prop}[Rough form of Proposition \ref{prop-st} and Corollary \ref{necc}]\label{intro-thm1}
In analogue with the ordinary $q$-expansion principle, $p$-adic automorphic forms in the $\mu$-ordinary setting are determined by their Serre--Tate expansions.  For forms $f_1, f_2$ of weights meeting appropriate conditions, $f_1\equiv f_2\mod p^n$ if and only if (sufficiently many of) the Serre--Tate expansions of their images in $V$ are congruent$\mod p^n$.  \end{prop}

In Sections \ref{structure-mo} and \ref{do-mo}, we construct the aforementioned differential operators. Our construction, which employs the Gauss--Manin connection, requires an appropriate replacement for the unit root splitting from \cite{kad} that Katz employs in his construction of differential operators in \cite{kaCM} (as the more general $\mu$-ordinary setting forces us to work with a module that is larger than just the unit root piece).  In the ordinary setting (e.g., \cite{kaCM, panch, apptoSHL, EFMV}), explicit description of the action of differential operators on $q$-expansions or Serre--Tate expansions allows one to construct $p$-adic families of automorphic forms.  In the $\mu$-ordinary setting, we expect the same should be true for Serre--Tate expansions.  We achieve a partial description of the action, and consequently families in Section \ref{new-families}, under certain conditions on the weights.  
\begin{thm}[Rough form of Theorem \ref{THETACONG}]\label{operators-thm}
For each positive dominant weight $\lambda$, there is a $p$-adic differential operator $\Theta^\lambda$ on $V$.  
At least under certain conditions on dominant weights $\lambda$ and $\lambda'$ (see Definition \ref{simple}), if $f$ is an automorphic form, then $\Theta^\lambda(f)\equiv\Theta^{\lambda'}(f)\mod p^n$ whenever $\lambda\equiv\lambda'\mod p^{n-1}(p-1)$.
\end{thm}
\noindent The computation of the action is substantially more challenging in the $\mu$-ordinary (with empty ordinary locus) setting, due to challenges coming from Lubin--Tate formal group laws, as seen in Section \ref{ST-mu-LT}.

\subsection{Challenges arising in the $\mu$-ordinary (but not ordinary) setting}\label{challenges-section}

While some aspects of Hida's theory carry over directly to the $\mu$-ordinary setting, obstacles arise for other aspects.  First, the replacement of the action of a general linear group on the ordinary Igusa tower by the action of its parabolic subgroup preserving the slope filtration leads to issues with multiplicities of representations, in turn forcing conditions on the weights of the automorphic forms embedding into $V$.  As seen in Section \ref{realization-section}, restricting to multiplicity-free weights does not by itself yield an embedding.  In fact, the sheaf of classical automorphic forms is not in general necessarily canonically isomorphic to the associate graded sheaf with which we must work over the $\mu$-ordinary Igusa tower.  By working locally, though, we are able to obtain results sufficiently strong for our applications.  Second, the construction of the differential operators in Sections \ref{structure-mo} and \ref{do-mo} requires intricate work (primarily by carefully extending \cite{KatzSlope, moonen}), especially for the Kodaira--Spencer morphism and an appropriate replacement for the unit root splitting, to accommodate the structure from the slope filtrations.  Furthermore, as noted above, the crucial description of the action of those operators requires involved formal group computations.  Finally, the argument used in \cite[Section 7]{EFMV} for constructing explicit families of automorphic forms on a product of unitary groups $G'$ embedded in a larger unitary group whose associated Shimura variety has nonempty ordinary locus falls apart whenever the ordinary locus of the Shimura variety associated to $G'$ is empty.

\subsection{Acknowledgements}
We thank H.\ Hida for feedback on the initial idea for this project.  We are also very grateful to E.\ de Shalit and E.\ Goren for helpful feedback on the first version of this paper, which led to key improvements.  
Our work on this paper benefitted from conversations with several additional mathematicians concerning topics closely related to this paper.  E.M.\ thanks B.\ Moonen for a discussion about his earlier work on which portions of this paper rely significantly, and she thanks M.-H.\ Nicole for a discussion about his work with W.\ Goldring on the $\mu$-ordinary Hasse invariant.    
We also thank E.\ Rains for alerting us to the Littlewood--Richardson rule mentioned in Section \ref{LRrule}, and we thank M.\ Harris for helpful feedback about references.  We are grateful to J.\ Fintzen and I.\ Varma for conversations concerning the collaboration \cite{EFMV}, whose influence is seen here.  We thank Caltech and the University of Oregon for hosting us during this collaboration.

\section{Preliminaries and Key Background Information}\label{preliminaries-section}
\subsection{Notation and conventions}\label{notation-section}

Let $p$ be an odd prime. 
Denote by $\BK$ the maximal unramified extension of $\Q_p$ in an algebraic closure $\bar{\Q}_p$ of $\Q_p$, $\W$ the completion of the ring of integers of $\BK$, $\F$ the residue field of $\W$, and $W(\F)$ the Witt vectors of $\F$.  We identify $\W$ with $W(\F)$.  For any integer $m\geq 1$, we write $\W_m=\W/p^m\W$, and for $S$ a $\W$-scheme, we write $S_m$ for the $\W_m$-scheme $S\times_{\spec \W}\spec {\W_m}$.

Let $F$ be a quadratic imaginary extension of a totally real field $F_0$ of degree $d$ over $\Q$.  We write $\T$ for the set of complex embeddings of $F$.  We write $\T_0$ for the set of real embeddings of $F_0$, as well as for a set containing a choice of an embedding in $\T$ above $\tau$ for each $\tau\in\T_0$  (i.e., for the set denoted $\Sigma$ in \cite{EFMV}).
We fix an isomorphism $\iota: \C\to \bar{\Q}_p$.  Via $\iota$, we view $\Dring:=\bar{\IQ}\cap\CO_{\bar{\IQ}_p}\subset\bar{\IQ}$ as a subring of $\IC$ and a subring of $\bar{\IQ}_p$. Via $\iota$, we also define a bijection $\tau\mapsto \iota\circ\tau$ between the complex embeddings $\tau:F\to \C$ and the $p$-adic embeddings of $F$ into $\bar{\Q}_p$.  

Assume $p$ is unramified in $F$. Then all $p$-adic embeddings of $F$ into $\bar{\Q}_p$ have image in $\BK$, and thus we can identify the set $\T$ with the set of embeddings of $\CO_F$ into $\W$, i.e., with the set ${\rm Hom}(\CO_F,\F)$.
Let $\sigma$ denote the Frobenius automorphism on $\F$.  (Abusing notation, $\sigma$ will also denote Frobenius on $\W$ and $\BK$.)  By composition, the identification $\T={\rm Hom}(\CO_F,\F)$ defines an action of $\sigma$ on $\T$. For each $\tau\in\T$ we write $\co_\tau$ for the orbit of $\tau$ under $\sigma$. We write $\FO$ for the set of $\sigma$-orbits $\co$ in $\T$.  Given a $\sigma$-orbit $\co$, we let $e_\co$ denote the cardinality of $\co$.

Note that there is a natural bijection between $\FO$ and the set of primes of $F$ above $p$.  For each prime $u|p$, we write $\co_u$ for the orbit associated with $u$, and $u_\co$ for the prime associated with an orbit $\co\in\O$.  For any $\co\in\FO$, we write $\co^*:=\co_{u^*}$ if $\co=\co_u$, where $*$ denotes the image under complex conjugation.  Finally, we define $\FO_0$ to be the subset of $\FO$ corresponding to a choice of a prime $u|v$ for each $v|p$ in $F_0$. Thus,  for $\FO_0^*:=\{\co^*|\co\in\FO_0\}$, we have
 $\FO=\FO_0\cup \FO_0^*$ (possibly not disjoint).  Note that $\FO_0:=\{\sigma_\tau|\tau\in\T_0\}$.

\subsection{Shimura varieties}\label{Shimura-section}

Following \cite[Section 5]{kottwitz}, we introduce the Shimura data and varieties with which we work.  Let $B$ be a simple $\Q$-algebra with center $F$.   Recall from Section \ref{notation-section} that $p$ is unramified in $F$.  We furthermore require that $B$ splits at each prime of $F$ above $p$.  Let $r$ be the rank of $B$ over $F$, and let $\CO_B$ be a $\ZZ_{(p)}$-order in $B$ whose $p$-adic completion $\CO_{B,p}$  is a maximal order in $B_{\IQ_p}$ and such that $\CO_{B,p}$ is identified with $M_r(\CO_{F, p})$.  
Let $*$ be a positive involution on $B$ over $\IQ$ preserving $\CO_B$.  
Let $(V, \langle, \rangle)$ consist of a finitely generated left $B$-module $V$ and a $\IQ$-valued hermitian form $\langle, \rangle$ on $V$, and let $G$ be the automorphism group of $(V, \langle, \rangle)$ (i.e., a general unitary group).  We also denote by $*$ the involution on $\End_B(V)$ coming from $\langle, \rangle$.  Let $h:\IC\rightarrow \End_B(V_\IR)$ be an $\IR$-algebra homomorphism such that $h(\bar{z}) = h(z)^*$ for all $z\in \IC$ and such that $(v, w)\rightarrow \langle v, h(i) w\rangle$ is positive definite on $V_\IR$.

As in \cite{kottwitz}, let $Sh(G,X)$ denote the unitary Shimura variety associated to the data $\CD=\left(B,V,*,\langle,\rangle, h\right)$, and let $E=E(G,X)$ denote the reflex field of $Sh(G,X)$.
We also assume that $p$ is a prime of good reduction for $Sh(G,X)$, i.e., that the level $\K$ of $Sh(G,X)$ is of the form $\K=\K^{(p)}\K_p$
where $\K^{(p)}\subset G(\A_f^p)$ is a neat, open compact subgroup, and $K_p$ is hyperspecial maximal compact.
We write $\Sh$ for the integral model of $Sh(G,X)$ over $\Z_p\otimes\CO_E$.

 Let $\iota_E: E\to \bar{\Q}_p$ denote the restriction of $\iota$ to $E\subset \C$. Our assumptions imply that 
$\iota_E$ factors through $\BK$,  mapping $\CO_E$ to $\W$. 
We write $\p$ for the associated prime above $p$ of $E$, $\CO_{E, \p}$ for the localization of $\CO_E$ at $\p$, $\CO_{E_\p}$ for the completion of $\CO_E$ at $\p$, and $\ck$ for its residue field. Via $\iota_E$, we regard $\CO_{E,\p}\subset\W$ and $\ck\subset \F$, and identify $\CO_{E_\p}=W(\ck)$.
Abusing notation, we still denote by $\Sh$ the base change of $\Sh$ to $\W$, and write $\sh$ for its special fiber. 

\subsection{Automorphic weights}\label{autweight}
Following \cite[Section 4]{kottwitz}, we write $V_\IC = V_1\oplus V_2$, where $V_1 = \left\{v\in V_\IC | h(z)v= zv \mbox{ for all } z\in\IC\right\}$ and $V_2 = \left\{v\in V_\IC | h(z) v = \bar{z} v \mbox{ for all } z\in \IC\right\}$.  Note that for $i=1, 2$, $V_i$ decomposes as 
\begin{align}\label{decompVi}
V_i = \oplus_{\tau\in\T}V_{i, \tau},
\end{align}
and $V_{\IC}$ decomposes as
\begin{align} 
V_{\IC}= \oplus_{\tau\in\T}V_\tau,
\end{align}
with
\begin{align}\label{decompVtau}
V_\tau := V_{1, \tau}\oplus V_{2, \tau}.
\end{align}

Let $\Levi$ be the Levi subgroup of $G_\IC^0$ that preserves this decomposition of $V_\IC$, i.e., the Levi subgroup determined by the signature $\left(a_\tau^+, a_\tau^-\right)$ of $G_\IC^0$ at each $\tau|_{F_0}$ for each $\tau\in\T$.  
Let $\Borel$ be a Borel subgroup of $\Levi$, let $T$ be a maximal torus contained in $\Borel$, and write $\Nilp$ for the unipotent radical of $\Borel$.  (Since $B$ is widely used in the literature to denote a Borel subgroup, as well as to denote a division algebra in the Shimura data, we will use $B$ for both of these.  Going forward, however, it will be clear from context which of these $B$ denotes, and indeed, it will soon be the case that we will only need to refer to the Borel and not the division algebra.)   Denote by $B^\op$ the opposite Borel with respect to $(B, T)$. We write $\Torus=\prod_{\tau\in\T}\Torus_\tau$, $\Borel=\prod_{\tau\in\T}\Borel_\tau$, and $\Borel^\op=\prod_{\tau\in\T}\Borel^\op_\tau$.    A choice of basis for $V_\IC$ that is compatible with the decompositions \eqref{decompVi}--\eqref{decompVtau} identifies $\Levi$ with $\prod_{\tau\in\T}\GL_{a_\tau^+}$.  Such a basis can be chosen so that $\Borel_\tau$ is identified with the subgroup of upper triangular matrices in $\GL_{a_\tau^+}$, $\Borel^\op_\tau$ with the subgroup of lower triangular matrices in $\GL_{a_\tau^+}$, and $\Torus_\tau$ with $\Torus_{a_\tau^+} := \Gm^{a_\tau^+}$, which is in turn identified with the subgroup of diagonal matrices in $\prod_{\tau\in\T}\GL_{a_\tau^+}$.  Via $\iota$, we can define $\Levi$ $p$-adically over $\CO_{E_\p}$.  Note that all our groups are split over $\CO_{E_\p}$.

Let $X^*(\Torus)$ denote the group of characters on $\Torus$.  For any module $M$ on which $\Torus$ acts, we denote by $M[\kappa]$ the $\kappa$-eigenspace for the action of $\Torus$ on the module $M$.

For the remainder of this subsection, we briefly recall key facts about algebraic representations of general linear groups and their relationships with certain characters on tori, following \cite[Section 2.4]{EFMV}, \cite[Sections 5.1.3 and 8.1.2]{Hida}, \cite[Part II, Chapter 2]{Jantzen}, and \cite[Sections 4.1 and 15.3]{FH}.  Let 
\begin{align*}
X(\Torus)_+ = \left\{\left(\kappa_{1, \tau}, \ldots, \kappa_{a_\tau^+, \tau}\right)_{\tau\in\T}\in\prod_{\tau\in\T}\ZZ^{a_\tau^+, \tau}| \kappa_{i, \tau}\geq \kappa_{i+1, \tau}\mbox{ for all } i \right\}.
\end{align*}
We identify $X(\Torus)_+$ with the set of {\it dominant weights} in the group $X^*(\Torus)$  of characters of $\Torus\subset \Borel$, given by 
\begin{align*}
\prod_{\tau\in\T}\diag\left(t_{1, \tau}, \ldots, t_{a_\tau^+, \tau}\right)\mapsto \prod_{\tau\in\T}\prod_{1\leq i \leq a_\tau^+} t_{i, \tau}^{\kappa_{i, \tau}}.
\end{align*}
(N.B. Such characters are also characters on $\Borel\supset T$, via $\Borel/\Nilp\cong T$.)
We write $\kappa$ both for an element of $X(\Torus)_+$ and for the associated character in $X^*(\Torus)$.  For each integer $k$, we write $\underline{k}$ for $\kappa$ such that $\kappa_{i, \tau} = k$ for all $i$ and all $\tau$.
We say that $\kappa$ is {\it positive} if $\kappa_{i, \tau}\geq 0$ for all $i$.  For positive $\kappa$, we define 
$d_{\kappa, \tau}:=\sum_{i=1}^{a_\tau^+}\kappa_{i, \tau}$ and $|\kappa|:=d_\kappa:=\sum_{\tau\in\T}d_{\kappa, \tau}$.  (Note that in \cite{EFMV}, we denote $d_{\kappa, \tau}$ by $d_{\kappa,\tau}^+$ or, equivalently, $d_{\kappa, \tau^*}^-$.)  
Similarly to \cite[Definitions 2.4.3 and 2.4.4]{EFMV}, we call a weight $\kappa =\left(\kappa_{1, \tau}, \ldots, \kappa_{a_\tau^+, \tau}\right)_{\tau\in\T} \in X(\Torus)_+$ {\it sum-symmetric} if $\kappa$ is positive and $d_{\kappa, \tau}=d_{\kappa, \tau*}$ for all $\tau\in\T$.  In this case, we call $e_\kappa:=d_\kappa/2$ the {\it depth} of $\kappa$ (or of the associated representation $\rho_\kappa$ introduced below).  If we furthermore have that $\kappa_{i, \tau}=\kappa_{i, \tau^*}$ for all $1\leq i\leq \min (a_\tau^+, a_{\tau^*}^+)$, then we call $\kappa$ {\it symmetric}.  Note that this is the same as the condition on the weights occurring in \cite[Theorem 12.7]{Shimura}.

Following the conventions of \cite[Section 2.4.2]{EFMV}, let $R$ be a $\ZZ_p$-algebra or a field of characteristic $0$, and for any dominant weight $\kappa$, let $\schur_\kappa$ denote the $\kappa$-Schur functor on the category of $R$-modules.  (A helpful reference on Schur functors is \cite[Section 15.3]{FH}.)  We denote by $\rho_\kappa=\rho_{\kappa, R}$ the representation $\schur_\kappa\left(\oplus_{\tau\in\T}\left(R^{a_\tau^+}\right)\right)$ of $\prod_{\tau\in\T}\GL_{a_\tau^+}$.   
If the ring $R$ is, furthermore, a field of characteristic $0$ (or of sufficiently large characteristic $p$), the algebraic representations $\rho_\kappa=\rho_{\kappa,F}$ of $\prod_{\tau\in\T}\GL_{a_\tau^+}$ are irreducible  and in bijection with the dominant weights $\kappa$ (see, e.g., \cite[Chapter II.2]{Jantzen}); and in the following, we refer to $\rho_\kappa$ as the irreducible representation of highest weight $\kappa$. 
When $R$ is such a field, the module $\rho_{\kappa,\CO}$ denotes our choice of an $\CO$-lattice inside the irreducible representation $\rho_{\kappa, R}$, where $\CO$ denotes the ring of integers in $R$.

Given a locally free sheaf of modules $\mathcal{F}$ and  $\kappa$ a dominant weight, we define
\begin{align*}
\mathcal{F}^\kappa:=\schur_\kappa\left(\mathcal{F}\right).
\end{align*}
Following the conventions of \cite{EFMV, EDiffOps, CEFMV}, we also sometimes write $(\cdot)^{\rho_\kappa}$ for $\schur_\kappa(\cdot)$,
and  $|\mathcal{F}|$ for the highest exterior power of $\mathcal{F}$.

Note that, more generally, we can replace $\Torus$ with any torus in a product of finite rank general linear groups and replace $X(\Torus)_+$ with ordered tuples on this torus and use Schur functors to construct representations in this context.   See \cite[Section 2.4.2]{EFMV} for a summary of the construction.  For example, we can consider $\Torus_{a_\tau^+}$ (respectively, $\Torus_{a_\tau^-}$) in $\GL_{a_\tau^+}$ (respectively, $\GL_{a_\tau^-}$).  If $\kappa_\tau$ is a positive dominant weight (ordered tuple, in this case) and $R$ is as above, of sufficiently large characteristic or of characteristic $0$, then the $\kappa_\tau$-Schur functor on the category of $R$-modules is $\schur_{\kappa_\tau}(V):= V^{\otimes d_{\kappa, \tau}}\cdot c_{\kappa, \tau}$, where $c_{\kappa, \tau}$ is the Young symmetrizer associated to $\kappa_\tau$.  Similarly to \cite[Section 2.4.3]{EFMV}, for each positive dominant weight $\kappa$ in $X(\Torus)_+$, by applying the generalized Young symmetrizer, we obtain a projection $\pi_\kappa: V^{\otimes d_\kappa}\rightarrow \rho_\kappa$, for $V$ the standard representation $\oplus_{\tau\in\T}\left(R^{a_\tau^+}\right)$ of $\prod_\tau\GL_{a_\tau^+}$.  If $\kappa$ is sum-symmetric of depth $e_\kappa$, then the representation $\rho_\kappa$ is a quotient of $\left(\otimes_{\tau\in\T}R^{a_\tau^+}\right)^{\otimes e_\kappa}$.

Following \cite[Definition 2.4.3]{EFMV}, which is inspired by \cite[Section 8.1.2]{Hida}, we write  $\cup_\tau\left\{b\dual_{1, \tau}, \ldots, b\dual_{a_\tau^+, \tau}\right\}$ for the dual basis to the standard basis $\cup_\tau\left\{b_{1, \tau}, \ldots, b_{a_\tau^+, \tau}\right\}$ of the standard representation $\oplus_{\tau\in\T}\left(R^{a_\tau^+}\right)$ of $\prod_{\tau\in\T}\GL_{a_\tau^+}$, and define $\ellcan^\kappa$ to be the basis of $\Hom_{B^\op}\left(\rho_\kappa, \kappa\right)$ such that 
$\ellcan^\kappa\circ\pi_\kappa =\prod_{\tau\in\T}\prod_{i=1}^{a_\tau^+}(\kappa_{i, \tau}!)^{-1}\cdot \otimes_{\tau\in\T}\otimes_{i=1}^{a_\tau^+}\left(b_{\tau, i}\dual\right)^{\otimes\kappa_{i,\tau}}\cdot c_\kappa$.  
We define $ {\tildeellcan}^\kappa:=\ellcan^\kappa\circ\pi_\kappa$.  By \cite[Lemma 2.4.6]{EFMV}, if $\kappa$ is a positive dominant weight and $\kappa'$ is a sum-symmetric weight, then $\pi_{\kappa\kappa'}$ factors through the map $\pi_\kappa\otimes\pi_{\kappa'}$, and ${\tildeellcan}^\kappa\otimes{\tildeellcan}^{\kappa'} = {\tildeellcan}^{\kappa\kappa'}.$

\subsection{Weights and representations}\label{littlewood-section}

Let $\Para$ denote a parabolic subgroup of $\Levi$ containing $\Borel$. We denote by $\Uni$ the unipotent radical of $\Para$, and we write $\Levii=\Para/\Uni$ for the Levi subgroup of $\Para$. 

For any irreducible algebraic representation $\rho $ of $\Levi$ over a field of characteristic $0$ (or of sufficiently large characteristic), the restriction of $\rho$ to $\Para$ admits a $\Uni$-stable filtration with irreducible $\Uni$-invariant subquotients.  Furthermore, after choosing a splitting $\Levii\subset \Para$, 
the associated graded representation  
$\gr(\rho{\vert_\Para})$ of $\Levii$  and $\rho\vert_{\Levii}$ are canonically identified.  We fix a splitting $\Levii\subset \Para$, and write $\Borell=\Borel\cap \Levii$, $\Nilpp=\Nilp\cap\Levii$, and $\Toruss=\Torus\cap \Levii$. Then $\Borell$ is a Borel subgroup of $\Levii$ with $\Nilpp$ its unipotent radical and $\Toruss=\Torus$ a maximal torus of $\Levii$.

\begin{defi}
Given dominant weights $\kappa$ of $\Levi$ and $\kappa'$ of $\Levii$, we say that $\kappa'$ divides $\kappa$ (and write $\kappa'\mid\kappa$) if the irreducible representation $\varrho_{\kappa'}$ of  $\Levii$, of highest weight $\kappa'$, arises as one of the 
irreducible constituents of $\rho_\kappa\vert_{\Levii}$.
 
For each dominant weight $\kappa$ of $\Levi$, we define $\M_\kappa:=\{\kappa' \mid\, \kappa' \mbox{ divides } \kappa\}$, regarded as a multi-set (so that we keep track of multiplicities).
Then
\begin{align}\label{rhokap-decomposition}
\rho_\kappa\vert_{\Levii}=\bigoplus_{\kappa'\in\M_\kappa}\varrho_{\kappa'}.
\end{align}
\end{defi}
Note that $\kappa^\sigma\mid \kappa$, 
 for any permutation $\sigma$ in the Weyl group of $\Levi$ such that $\kappa^\sigma$ is dominant for $\Levii$. 
In particular, for any dominant weight $\kappa$ of $\Levi$, $\kappa$ is also a dominant weight of $\Levii$ and $\kappa \mid \kappa$. Note that if $\kappa $ is a scalar weight of $\Levi$, then
the only dominant weight of $\Levii$ dividing $\kappa$ is $\kappa$ itself, i.e., $\M_\kappa=\{\kappa\}$ and  $\rho_\kappa\vert_{\Levii}=\varrho_\kappa$.

\subsubsection{Littlewood--Richardson rule}
\label{LRrule}

In general, the multiplicity $c_{\kappa,\kappa'}$ of the irreducible constituent $\varrho_{\kappa'}$ in $\rho_{\kappa}{\vert_{\Levii}}$ can be explicitly computed using the Littlewood--Richardson rule (\cite[Rule (9.2) in Section 9]{littlewood-richardson} or \cite[Equation (A.8)]{FH}).

\begin{defi}
We say that a dominant weight $\kappa$ of $\Levi$ is multiplicity-free (with respect to the Levi subgroup $\Levii$) if $c_{\kappa,\kappa'}=1$ for all $\kappa'|\kappa$. 
\end{defi}

We observe that as $\kappa$ varies among the dominant weights of $\Levi$, the sets $\M_\kappa$ are not necessarily disjoint. 

If $|\kappa_1|\neq |\kappa_2|$, then by considering the central action of the scalars, we see that
$\kappa_1$ and $\kappa_2$ are coprime, i.e., $\M_{\kappa_1}\cap\M_{\kappa_2}=\O$.
More generally, given any two dominant weights $\kappa_1,\kappa_2$ of $\Levi$, the Littlewood--Richardson rule allows one to compute the intersection $\M_{\kappa_1}\cap \M_{\kappa_2}$.

\begin{rmk}\label{lattices}
Assume $\Levii\subset \Levi$ is a split Levi subgroup, defined over $\CO_{E_\p}$. 
Then the Schur projectors and Young symmetrizers from Section \ref{autweight} are compatible with the decomposition in (\ref{rhokap-decomposition}), in the sense that
$\schur_\kappa(V){\vert_{\Levii}} = \oplus_{\kappa'\in\M_\kappa}\schur_{\kappa'}(V)$, for $V$ the standard representation, and $R=E_\p$.

Furthermore,  our choices of $\CO_{E_\p}$-lattices of the irreducible algebraic representations $\rho_\kappa, $ and $\varrho_{\kappa'}$,  for $\kappa,\kappa'$ dominant weights of $\Levi,\Levii$, are also compatible with the decomposition in (\ref{rhokap-decomposition}), in the sense that the Schur projectors induce a morphism $\rho_{\kappa,\CO_{E_\p}}\vert_{\Levii}\hookrightarrow\oplus_{\kappa'\in\M_\kappa}\varrho_{\kappa',\CO_{E_\p}}$, which is an isomorphism if $p$ is sufficiently large (or after inverting $p$).
\end{rmk}

\subsection{Automorphic sheaves}\label{autsheaf}

Let $A$ denote a $\CD$-enriched abelian scheme (i.e., an abelian scheme with additional structures defined by the Shimura data $\CD)$ over a $\W$-scheme $S$. (For example, the abelian scheme $A/S$ could be the universal abelian scheme $\CA$ over $\Sh$.)  Then the Dieudonn\'e crystal of $A$ decomposes according to the embeddings $\tau\in \T$ and Morita equivalence (via $\CO_B\otimes_{\CO_F,\tau}\W\simeq M_r(\W)$) as 
$$H^1_{dR}(A):=H^1_{dR}(A/S)=\oplus_{\tau\in\T} M_\tau^{\oplus r}.$$  Similarly, the Hodge filtration $\omega(A):=Fil^1(H^1_{dR}(A))$ decomposes as
$$\omega(A)=\oplus_{\tau\in\T} \omega_\tau^{\oplus r}.$$
Note that $M_\tau$ and $\omega_\tau$ are locally free.  Note also that, for each $\tau\in\T$, the rank of $M_\tau$ is independent of $\tau$, while the rank of $\omega_\tau$ depends on $\tau$.
More precisely, if $(a^+_\tau,a^-_\tau)$ denotes the signature of $G$ at $\tau_{|F_0}$, then we have
$\rk M_\tau=n=a^+_\tau+a^-_\tau$ and $\rk \omega_\tau=a^+_\tau$.
Following the notation in \cite{moonen}, we write
$\cf(\tau)=\rk \omega_\tau$.  Observe that $\cf\left(\tau^*\right)=a^+_{\tau^*} = a^-_{\tau}$.

\subsubsection{Classical automorphic forms}Similarly to \cite[Section 2.3]{EFMV} and \cite[Section 3.2]{CEFMV}, we now define classical automorphic forms.  First, we define a sheaf
\begin{align*}
\CE_{A/S}:= \prod_\tau\Isom\left(\omega_\tau, \CO^{\cf(\tau)}_{S, \tau}\right)=:\Isom_{\CO_{B, (p)}\otimes \CO_S}\left(\omega, \CO^{g}_{S}\right),
\end{align*}
where $\CO_{B, (p)}$ is the localization of $\CO_B$ at $(p)$ and $g=\rk(\omega)$.
When it is clear from context, we drop the subscript $A/S$ and just write $\CE$.
Note that there is a left action of $\Levi$ on $\CE_{A/S}$ coming from the action, for each $\tau\in\T$, of $\GL_{a_\tau^+}$ on $\Isom\left(\omega_\tau, \CO^{\cf(\tau)}_{S, \tau}\right)$.  Following \cite[Section 2.3]{EFMV}, for any representation $\left(\rho, M_\rho\right)$ of $\Levi$, we define the sheaf
\begin{align*}
\CE^\rho:=\CE_\rho:=\CE\times^\Levi M_\rho
\end{align*}
so that for each open immersion $\Spec R\hookrightarrow S$, $\CE^\rho(R):=\left(\CE(R)\times M_\rho\otimes R\right)/(\ell, m)\sim(g\ell, \rho({ }^tg^{-1})m)$.  

An {\it automorphic form (defined over $R$) of weight $\rho$} is then a global section of $\CE^\rho$ on $\Sh_R$.  An automorphic form of weight $\rho$ and {\it level} $\K$ is a global section of $\CE^\rho$ on $\Sh_R$ with $\K$ the level of $\Sh_R$.  We exclude the case in which $F_0=\IQ$ with $a_\tau^+ = a_\tau^-=1$.  (As far as this paper's goals are concerned, nothing is lost from this exclusion.  The interesting cases in this paper concern unitary groups of higher rank.)  Then by the Koecher principle (stated in great generality in \cite[Theorem 2.5]{lan5}, with additional details of interest in our setting in \cite[Theorem 2.3 and Section 10]{lan5}), our space of automorphic forms is the same as the space we would have obtained by instead working over a compactification of our moduli space.

It follows from the definitions that for $\kappa$ a dominant weight of $\Levi$, and $\rho=\rho_\kappa$
our choice of $W(\ck)$-lattice of the irreducible algebraic representation of $\Levi$ of  highest weight $\kappa$, the sheaves $\omega^\kappa$ and $\CE^\rho$ are canonically identified.
In the following, we prefer the notation $\omega^\kappa$ to $\CE^\rho$.

\subsection{The $\mu$-ordinary locus} 
We maintain the notation of Section \ref{autsheaf}.

Let $A$ denote a $\CD$-enriched abelian scheme (i.e., an abelian scheme with additional structures defined by the data $\CD$, e.g., PEL structure) over a smooth $\F$-scheme $S$ (e.g., the universal abelian scheme $\CA$ over $\sh$). We write $\Phi$ for the Frobenius map on the filtered Dieudonn\'e crystal of $A$, 
$$\omega(A)=\oplus_{\tau\in\T} \omega_\tau^{\oplus r}\subset H^1_{dR}(A)=\oplus_{\tau\in\T} M_\tau^{\oplus r}.$$

For each $\tau\in \T$, $\Phi$ induces a map $\Phi_\tau:M_\tau\to M_{\tau\circ \sigma}$. In particular, we deduce that, for $e_\tau$ the cardinality of $\co_\tau$,  the pair $(M_\tau,\Phi_\tau^{e_\tau})$ is a Dieudonn\'e crystal, whose isogeny class depends only on the orbit $\co_\tau$ of $\tau$.  In the following, we write $\nu_\tau$ (or $\nu_\co$ for $\co=\co_\tau$) for the Newton polygon of $(M_\tau,\Phi_\tau^{e_\tau})$.

\begin{rmk}\label{BTslope}
The above description reduces the computation of the Newton polygon of the Frobenius map $\Phi$ on $H^1_{dR}(A)$ to that of the polygons $\nu_\co$, for $\co\in\FO$. 
By abuse of notation, in the following, we refer to the slopes of $\nu_\co$ as the slopes at $\co$ of $(H^1_{dR}(A),\Phi)$.
\end{rmk}

The notion of $\mu$-ordinariness is originally due to Wedhorn in \cite{wedhorn}.  The following definition is adapted to the above notation, and follows  \cite[Section 1.2.5]{moonen}.  (In {\it loc. cit.} the $\mu$-ordinary Newton polygon $\nu_\co(n,\cf)$ is denoted by ${\rm Ord}(d,\cf)$, for $n=d$.)

\begin{defi}\label{slopes}
For each $\sigma$-orbit $\co$ in $\mathfrak O$, we define $\cf=\cf_\co:\co\to \{1, \dots ,n\}$ by
$\cf(\tau):=\rk \omega_\tau$, for  $\tau\in\co$.  (The function $\cf$ is called the {\it multiplicative type}.)
We define the {\it $\mu$-ordinary Newton polygon} $\nu_\co(n,\cf)$ associated with the triple $(\co,n,\cf)$ 
to be the polygon with slopes \[a^\co_j:=\#\{\tau\in \co |\cf(\tau)>n-j\},\]
for $j=1,\dots, n$.
\end{defi}

\begin{defi}
A $\CD$-enriched abelian variety $A$ over a field containing $\F$ is called $\mu$-ordinary if for each $\tau\in\T$ the associated Newton polygon $\nu_\tau$ agrees with the $\mu$-ordinary polygon $\nu_{\co_\tau}(n,\cf)$.
\end{defi}

We say that a point $x$ of $\sh$ is $\mu$-ordinary if the associated $\CD$-enriched abelian scheme $\CA_x$ is $\mu$-ordinary.
In the following, we denote by $\sh^{\muord}\subseteq \sh$  the $\mu$-ordinary locus of $\sh$.

In \cite{wedhorn}, Wedhorn proves that the $\mu$-ordinary locus is the largest nonempty Newton stratum of $\sh$.

\begin{thm} \cite[(1.6.2) Density Theorem]{wedhorn}
The Newton stratum $\sh^{\muord}$ is open and dense in $\sh$. In particular, it is nonempty.
\end{thm}

Subsequently, in \cite{moonen}, Moonen gives an explicit construction of a 
$\CD$-enriched \BT group $\X=\X^{\muord}(\CD)/\ck$ in terms of the triples $\{(\co_\tau, n, \cf)\}_{\tau\in\T}$, and he proves that the $\mu$-ordinary locus is also the largest Ekedahl--Oort stratum of $\sh$, and the central leaf associated with $\X^{\muord}(\CD)/\F$ (recall $\ck\subset \F$). That is, he proves the following result.

\begin{thm}\cite[Theorem 3.2.7]{moonen}
Let $A$ be a $\CD$-enriched abelian variety over an algebraically closed field containing $\F$.
Then the following are equivalent:
\begin{itemize}
\item $A$ is $\mu$-ordinary (equivalently, $A[p^\infty]$ is isogenous to $\X$ as $\CD$-enriched \BT groups).
\item $A[p]$ is isomorphic to $\X^{\muord}(\CD)[p]$ as $\CD$-enriched truncated \BT groups of level 1.
\item $A[p^\infty]$ is isomorphic to $\X^{\muord}(\CD)$ as $\CD$-enriched \BT groups.
\end{itemize}
\end{thm}

\subsection{The $\mu$-ordinary Hasse invariant}\label{hasse}\label{HasseGN}

Building on Moonen's work, in \cite{GN} Goldring and Nicole construct a $\mu$-ordinary Hasse invariant.  

Let $\pi:\CA\to \Sh$ denote the universal abelian scheme over $\Sh$, and let $|\omega|$ denote the Hodge line bundle over $\Sh$, i.e., 
$|\omega|=\wedge^{\rm top} \pi_*\Omega^1_{\CA/\Sh}$, where $\wedge^{\rm top}$ denotes the top exterior power.

\begin{thm}\cite[Theorem 1.1]{GN}
There exists an explicit positive integer $m_0\geq 1$, and a section $$\hasse\in H^0(\sh,|\omega|^{m_0})$$
such that: 
\begin{enumerate}
\item The non-vanishing locus of $\hasse$ is the $\mu$-ordinary locus of $\sh$.
\item The construction of $\hasse$ is compatible with varying the level $\K^{(p)}$.
\item The section $\hasse$ extends to the minimal compactification of $\sh$.
\item A power of $\hasse$ lifts to characteristic zero.
\end{enumerate}
\end{thm}

By construction  (\cite[Definition 3.5]{GN}),
$m_0={\rm lcm}_{\tau\in\T}(p^{e_\tau}-1)$.
In the following, for convenience, we replace $\hasse$ with one of its powers which lifts to characteristic zero. We choose $\hasse\in H^0(\Sh^{\min},|\omega|^{m})$, for some $m\geq 1$ a  multiple of $m_0$, where $\Sh^{\min}$ denotes the minimal compactification of $\Sh$.

\begin{defi} 
We define the $\mu$-ordinary locus  $\Sh^{\muord}_{/\W}$ (respectively,  $\Sh^{\min,\muord}_{/\W}$)  as the locus   in $\Sh_{/\W}$ (respectively, $\Sh^{\min}_{/\W}$) where $E_\mu$ is invertible.
\end{defi}

Similarly to the treatment of the ordinary case in \cite[Section 8.1]{Hida}, 
we define the formal scheme $\Sh_\infty^{\muord}$ over $\W$ as the injective limit  of the schemes $\Sh_m^{\muord}/\W_m$.
Note that
$\Sh_\infty^{\muord}$ is  the formal completion of $\Sh^{\muord}$ along its special fiber modulo $p$,  $\Sh^{\muord}_1=\sh^{\muord}$.

We observe that, by construction, the sheaf $|\omega|^{m_0}$ is trivial on the $\F$-scheme $\sh^{\muord}$.
We normalize $\hasse$ so that on $\sh^{\muord}$ $$\hasse\equiv 1\mod p,$$ and we call $\hasse$ the {\it $\mu$-ordinary Hasse invariant}.
In the following, for simplicity, we set $\cS:=\Sh^{\muord}_\infty$ and $\bS:=\sh^{\muord}$.

\subsection{The $\mu$-ordinary \BT group} In this section, we briefly recall the construction of the $\mu$-ordinary $\CD$-enriched \BT group $\X=\X^{\muord}(\CD)/\ck$ from Moonen \cite[Section 1.2.3]{moonen}

Let $N^{\muord}(\CD)$ denote the Dieudonn\'e crystal of $\X$. The underlying $\W$-module $N^{\muord}(\CD)$ decomposes  according to the embeddings $\tau\in \T$. Grouping together the submodules corresponding to $\tau\in\co$, for each orbit $\co$ in $\FO$, we obtain a decomposition in subcrystals 
$$N^{\muord}(\CD)=\oplus_{\co\in\FO} N(\co,n,\cf)^{\oplus r}.$$
We write the associated decomposition of $\X^{\muord}(\CD)$ as  $$\X^{\muord}(\CD)=\oplus_{\co\in\FO}\X(\co,n,\cf)^{\oplus r}.$$

Note that the $\CD$-enriched structure on $N^{\muord}(\CD)$ (respectively, $\X^{\muord}(\CD)$) induces a structure of  $\CO_{F,u_\co}$-crystals on $N(\co,n,\cf)$ (respectively, of 
$\CO_{F,u_\co}$-modules on $\X(\co,n,\cf)$).

Fix an orbit $\co\in\FO$.  Following the conventions of Section \ref{notation-section}, let $e_\co$ denote the cardinality of $\co$.  Let  $0\leq a^\co_1\leq \dots \leq a^\co_n\leq e=e_\co$ denote the slopes of the $\mu$-ordinary polygon $\nu_\co(\cf,n)$ introduced in Definition \ref{slopes}.  
We write 
$0\leq \lambda_0< \cdots <\lambda_s$,  $s=s_\co$, for the (distinct) integers occurring as slopes $a^\co_j$ for some $j$, $1\leq j\leq n$.
For each $t=0,\dots ,s$, we denote by $m_t = m_t^{\co_\tau}$ the multiplicity of the slope $\lambda_t$, i.e., \[m_t:=\#\{j\in\{1, \ldots, n\}|a^\co_j=\lambda_t\}.\]
Note that $\sum_0^s m_t=n$.
\begin{defi}
The crystal $N(\co,n,\cf)$ is defined as   \[N(\co,n,\cf):=N(\lambda_0)^{\oplus m_0}\bigoplus \cdots \bigoplus N(\lambda_s)^{\oplus m_s},\] where for each $t=0,\dots ,s$, the crystal  $N(\lambda_t)$ is the simple isoclinic $\CO_{F,u_\co}$-crystal of slope $\lambda_s$, and height (i.e., rank) $e$.
We write the associated decomposition of $\X(\co,n,\cf)$ into isoclinic components as \[\X(\co,n,\cf):=\X(\lambda_0)^{\oplus m_0}\bigoplus \cdots \bigoplus \X(\lambda_s)^{\oplus m_s}.  \] 
\end{defi}

We observe that, for $\co\neq \co^*$, the polarization of $\X$ induces an isomorphism of $\CO_{F,u_\co}$-crystals between  $N(\co,n,\cf)$ and the  dual of $N(\co^*,n,\cf)$. In particular, $\lambda$ is a slope of $N(\co,n,\cf)$  with multiplicity $m$ if and only if $1-\lambda$ is a slope of $N(\co^*,n,\cf)$
with the same multiplicity. 
For $\co=\co^*$,  the polarization of $\X$ induces a  polarization on $N(\co,n,\cf)$, i.e., $N(\co,n,\cf)$ inherits the structure of a $(\CO_{F,u_\co},*)$-crystal. In particular,  $\lambda$ is a slope of $N(\co,n,\cf)$  with multiplicity $m$ if and only if $1-\lambda$ is also a slope with the same multiplicity.

In \cite[Definition 2.3.10]{moonen}, Moonen defines a canonical lifting $\X^{\rm can}=\X^{\rm can}(\CD)$ of $\X$ over $W(\ck)$.  Concretely, $\X^{\rm can}:=\oplus_\co \X(\co,n,\cf)^{\rm can}$
where \begin{align*}
\X(\co,n,\cf)^{\rm can}:= \left( \X^{\rm can}(\lambda_0)^{\oplus m_0}\bigoplus \cdots \bigoplus \X^{\rm can}(\oplus\lambda_s)^{m_s}\right)^{\oplus r}
\end{align*} 
and, for each  $t=0,\dots,s$,
$\X^{\rm can}(\lambda_t)$ is the unique lifting of  the $\CO_{F,u}$-module $\X(\lambda_t)$ (\cite[Corollary 2.1.5]{moonen}).
The $\CD$-enriched \BT group $\X^{\rm can}$ is characterized by the following property.

\begin{prop}\label{canonicallift}
\cite[Proposition 2.3.12]{moonen} The canonical lifting $\X^{\rm can}$  is the unique lifting  of $\X$ with the property that (geometrically) all endomorphisms lift.  
\end{prop}

\subsection{The $\mu$-ordinary Levi subgroup}\label{allsubgr}
In this section, we introduce a Levi subgroup $\levi$ of $G^0/\Q_p$ associated with the $\mu$-ordinary polygon which plays a crucial role in our results.  As highlighted in Remark \ref{details} below, the group $\levi$ arises as a subgroup of the Levi subgroup $\Levi$ introduced in Section \ref{autweight}.

\begin{defi}
We define $\levi$ to be the algebraic group over $\Q_p$ of automorphisms of the $\CD$-enriched isocrystal $N^{\muord}(\CD)[\frac{1}{p}]$.

In particular, $$\levi(\Q_p)={\rm Aut}^0_{\CD}(\X/\F),$$ 
the group of non-zero quasi-self-isogenies of the $\CD$-enriched \BT group $\X/\F$.

\end{defi}
\begin{rmk}
The algebraic group $\levi$ arises as a Levi subgroup of $G^0/\Q_p$. More precisely,  it is the Levi subgroup 
associated with the partitions of $n$  defined by the multiplicities $\{m_{s}^\tau,\dots , m_0^\tau\}_{\tau\in\T}$ of the slopes of $N^{\muord}(\CD)$ (here, $m_t^\tau:=m_t^{\co_\tau}$, $s:=s_{\co_\tau}$).  
\end{rmk}

More explicitly,
the decomposition of
$N^{\muord}(\CD)$ as a sum of the subcrystals $N(\co,n,\cf)$, $\co\in\FO$, induces the decomposition
$$\levi=\prod_{\co\in\FO_0} \levi(\co),$$ 
where, for each $\co\in\FO_0$, ${\levi}(\co)$ is the algebraic group of automorphisms of the $\CO_{F,u_\co}$-isocrystal $N(\co,n,\cf)[\frac{1}{p}]$, for $\co\neq\co^*$, and of the polarized $\CO_{F,u_\co}$-isocrystal $N(\co,n,\cf)[\frac{1}{p}]$, for $\co=\co^*$.

Following \cite[Lemma 1.3.11]{moonen}, if we write $\kappa(\co)=\kappa(u_\co)$ for the residue field of $F_{u_\co}$, then for $\co\neq \co^*$, we have 
$$\levi (\co)(\Q_p)= GL_{m_s^\co}(W(\kappa(\co)))\times\cdots \times  \GL_{m_0^\co}(W(\kappa(\co))),$$
and for $\co=\co^*$, assuming $e/2$ is not a slope, we have that the number of slopes $s+1$ is even and
$$\levi (\co)(\Q_p)= GL_{m_s^\co}(W(\kappa(\co)))\times\cdots \times  \GL_{m_{\frac{s+1}{2}}^\co}(W(\kappa(\co))).$$
As alluded to in the first paragraph of Section \ref{mr-section}, for notational convenience, {\bf we exclude the slope $e/2$, but we expect no mathematical issues extending to this case.}

\begin{rmk}\label{details}

Note that $\levi$ is defined over $\Z_p$, while $\Levi$ is defined over $\CO_{E,\p}$ (as in Section \ref{autweight}).  It follows from the definitions of these groups that after base change, the group $\levi$ is contained in $\Levi$, with equality exclusively when the $\mu$-ordinary polygon is ordinary (i.e., when $\CO_{E,\p}=\ZZ_p$).

More explicitly, let $F_1>\cdots > F_{s}$ denote the distinct values of $\cf(\tau)$, for $\tau\in\co$, in the interval $[1,n-1]$ ($s=s_\co\geq 0$). For convenience, we also write $F_0:=n$ and $F_{s+1}:=0$. 
For each $i=0, \dots, s+1$, we define \[d_i:=\{\tau\in\co|\cf(\tau)=F_i\}.\] Note that $d_i>0$ for all $i=1,\dots, s$, and $d_0, d_{s+1}\geq 0$. Note that $e=e_\co=\sum_{i=0}^{s+1} d_i$.

With this notation, the distinct slopes of the $\mu$-ordinary polygon $\nu_\co(n, \cf)$, associated with the orbit $\co$, are 
\[\lambda_i=\sum_{j=0}^id_j,\]
where each $\lambda_i$ occurs in $\nu_\co$ with multiplicity $m_i=F_i-F_{i+1}$, $i=0,\dots ,s$. 

Note that, for each $\tau\in\co$, $\cf(\tau)=\sum_{i_\tau}^s m_j$, where the integer $i_\tau$, $0\leq i_\tau\leq s_\tau$, is defined by the condition $\cf(\tau)=F_{i_\tau}$.
\end{rmk}

\begin{defi}
We define $\para$ to be the parabolic subgroup of $\Levi$ that contains the chosen Borel subgroup $\Borel$ and has Levi subgroup $\levi$ associated with the ordering on the partitions of $n$ defined by the decreasing ordering of the slopes of $N^{\muord}(\CD)$.  We write $\uni$ for the unipotent radical of $\para$.
\end{defi}

\begin{rmk}\label{rmk-on-levi}
One can choose a basis as in Section \ref{autweight} that (in addition to the identifications in Section \ref{autweight}) identifies $\para$ with a subgroup of block upper triangular matrices of $\prod_{\tau\in\T}\GL_{a_\tau^+}$ and $\levi$ with a subgroup of block-diagonal matrices.  
\end{rmk}

We define $\borel:=\Borel\cap\levi$ and $\nilp:=\Nilp\cap\levi$. Under our assumptions, $\borel$ is a Borel subgroup of $\levi$, $\nilp$ its unipotent radical, and  the maximal torus $\torus$ of $\levi$ in $\borel$ is also the maximal torus $\Torus$ of $\Levi$ contained in $\Borel$.

\begin{rmk}
If a weight $\kappa$ of $\Torus=\torus$ is dominant in $X^*(\Torus)$, then it is also dominant in $X^*(\torus)$, but the converse does not hold in general.
\end{rmk}

\section{The $\mu$-ordinary Igusa Tower}\label{tower-section}
In this section, we introduce basic details of the $\mu$-ordinary Igusa tower over the $\mu$-ordinary locus, building on \cite{moonen, mantovan1}.  By \cite{wedhorn}, assuming $p$ is unramified in the reflex field $E$, the $\mu$-ordinary locus is always nonempty.  In the case where the ordinary locus is nonempty (i.e., when $p$ splits completely in $E$), the $\mu$-ordinary Igusa tower coincides with Hida's ordinary Igusa tower.  
\subsection{The $\mu$-ordinary slope filtration} 
Let $H$ be a $\CD$-enriched \BT group over a smooth $\F$-scheme $S$ (e.g., $H=\CA[p^\infty]$ for $\CA$ the universal $\CD$-enriched abelian scheme over $\sh$).  
The $\CD$-structure on $H$ induces a decomposition of $H$ according to the primes $u$ of $F$ above $p$ and Morita equivalence (via $\CO_B\otimes_{\CO_F}\CO_{F,u}\simeq M_r(\CO_{F,u})$).  That is, we have 
$$H=\bigoplus_{u|p} H[u^\infty],$$
and for each $u|p$, we have
$$ H[u^\infty]=G(u)^{\oplus r},$$
where  $G(u)$ is a \BT $\CO_{F,u}$-module  (\cite[Sections 3.1.2, and 4.1.3]{moonen}). More precisely, for each $u|p$, the $\CD$-structure of $H$ induces a structure of $\CO_{F,u}$-modules on $G(u)$, together with an  isomorphism between $G(u)$ and the Cartier dual of $G({u}^*)$, for $u\neq u^*$,  and a structure of \BT $(\CO_{F,u},*)$-modules on $G(u)$, for $u=u^*$. 
If we write the Dieudonn\'e crystal of $H$  as $\bD(H)=\oplus_{\tau\in \T} M_\tau^{\oplus r}$, then for each prime $u|p$,  the Dieudonn\'e crystal of $G(u)$ is the subcrystal $M_\co=\oplus_{\tau\in\co} M_\tau$ of $\bD(H)$, for $\co=\co_u$.

In \cite[Theorem 7]{zink}, Zink proves that any \BT group over a regular scheme with constant Newton polygon is isogenous to a completely slope divisible \BT group, i.e., to a \BT group that has a slope filtration with slope divisible quotients (see \cite[Definition 10]{zink}).  
In the case of the universal \BT group over central leaves $C$ in Oort's foliation, Zink's result can be strengthened to prove that the restriction of the universal \BT group $\CA[p^\infty]$ to $C$ is slope divisible.

\begin{prop}\cite[Section 3]{mantovan1} (see also \cite[Section 3.2.3]{mantovanA})\label{csd}
Let $C\subset \sh$ be a central Oort's leaf, i.e., an Oort's leaf associated with a completely slope divisible \BT group. Then the restriction to $C$ of the universal \BT group $\CA[p^\infty]$  is completely slope divisible.
\end{prop}

\begin{rmk}
Since the above decomposition of $H$ is canonical,  we deduce that $H$ is isomorphic to $\X^{\muord}(\CD)$ as $\CD$-enriched \BT groups if and only if for each prime $u|p$
the \BT groups $G(u)$ 
 and  $\X(\co_u,n,\cf)$ are isomorphic as $\CO_{F,u}$-modules for $u\neq u^*$ and as $(\CO_{F,u},*)$-modules for $u=u^*$ (\cite[Proposition 2]{mantovan1}).
In particular, for $H=\CA[p^\infty]$, we deduce that for each $u|p$ $\CA[u^\infty]=\G(u)^{\oplus r}$ where $\G(u)$ is an $\CO_{F,u}$-module whose restriction  to any central Oort's leaf is completely slope divisible.
\end{rmk}

Let $\cS:=\Sh_\infty^{\muord}$  denote the formal $\mu$-ordinary locus over $\W$, and $\bS:=\sh^{\muord}$ the $\mu$-ordinary locus of $\sh$ over ${\mathbb F}$ (see section \ref{HasseGN} for definitions). 
We apply the above result to the case of $C=\bS$.
For each $u|p$, we write $\G(u)_\bullet$ for the  slope filtration of $\G(u)$ over $\bS$, and $\gr(\G(u))$ for the \BT $\CO_{F_u}$-module defined as the direct sum of the associated subquotients. Similarly, we write $\CA[p^\infty]_\bullet$ for the slope filtration of $\CA[p^\infty]$ over $\bS$ and  $\gr(\CA[p^\infty])$ for the direct sum of its subquotients. Thus, $\gr(\CA[p^\infty])=\oplus_{u|p} \gr(\G(u))^{\oplus r}$. 
Note that if $u\neq u^*$, then the polarization induces an isomorphism of $\CO_{F_u}$-modules  between $\gr(\G(u))$ and $\gr(\G(u^*))$.
For  $u=u^*$, the \BT group $\gr(\G(u))$ is a polarized $(\CO_{F_u},*)$-module, which arises as the direct sum of pairs of dual isoclinic $\CO_{F_u}$-modules (namely, the two subquotients of $\G(u)$ of slope $\lambda, 1-\lambda$, for $\lambda\neq 1/2$) and (possibly) of a polarized isoclinic $(\CO_{F_u},*)$-module of slope $1/2$.
Then $\gr(\CA[p^\infty])$ is a naturally $\CD$-enriched \BT group, and for each $u|p$ $\gr(\CA[u^\infty])=\gr(\G(u))^{\oplus r}$.

\begin{rmk}
It follows from the fact that the slope filtration is canonically split over perfect fields, that at all geometric points $x$ of $\bS$, the $\CD$-enriched \BT groups  $\gr(\CA[p^\infty])_x$ and  $\CA[p^\infty]_x$  are canonically isomorphic. In particular, $\gr(\CA[p^\infty])_x$ is  isomorphic to $\X^{\muord}(\CD)$ for 
all points $x$ of $\bS$.
\end{rmk}

\begin{prop}\label{slopefiltration}
Maintain the above notation.
The slope filtration of $\CA[p^\infty]$ over $\bS$ canonically lifts to $\cS$. 
\end{prop}
\begin{proof}
The lifting of the slope filtration to $\cS$ is a consequence of \cite[Proposition 2.1.9]{moonen} given the existence of the filtration on $\bS$ (Proposition \ref{csd}). 
\end{proof}

\begin{rmk}
It follows from the rigidity of isoclinic $\CD$-enriched \BT groups (\cite[Proposition 2.1.5]{moonen}) that $\gr(\CA[p^\infty])$ over $\cS$ is naturally a $\CD$-enriched \BT group. Furthermore, for all geometric points $x\in \bS$, the $\CD$-enriched \BT group $\gr(\CA[p^\infty])$ restricted to $\cS^\wedge_{x}$, the formal completion of $\cS$ at $x$, is isomorphic to $\X^{can}$.
\end{rmk}

\subsection{$\mu$-ordinary Igusa tower}
Maintaining the above notation, we introduce the formal $\mu$-ordinary Igusa tower $\Igmu$ as a profinite \'etale cover of the formal $\mu$-ordinary locus $\cS$.
Let $\gr(\CA[p^\infty])/\cS$ be the $\CD$-enriched \BT group introduce above.
In the following, we write $\gr(\CA[p^n])=\gr(\CA[p^\infty])[p^n]$, for all $n\geq 1$.

\begin{prop}
For each $m,n\geq 1$, we define $\Igmunm$, the Igusa cover of level $n$ over $\cS_m$, to be the $\W_m$-scheme $${\Igmunm}:={\underline{\Isom}}_{\CD}(\X^{\can}(\CD)[p^n], \gr(\CA[p^n]/\cS_m)).$$

The space $\Igmunm$ is a finite \'etale cover of $\cS_m$ with Galois group $\levi(\Z/p^n\Z)$.
\end{prop}
\begin{proof}
For  $m=1, n\geq 1$ the statement is proved in \cite[Proposition 4]{mantovan1}.  A similar proof applies for all $m, n\geq 1$. 
Indeed, for each $m\geq 1$, the $\W_m$-scheme $\Igmunm$ (respectively, the formal $\W$-scheme $(\Igmu)_n$) is the unique finite \'etale cover of $\cS_m$ (respectively, of $\cS$) with reduced fiber $(\Igmu)_{n,1}/\cS_1=\bS$.
\end{proof}

\subsection{Irreducibility of the Igusa tower} A key result in Hida's theory is the irreducibility of the Igusa tower.  To be exact, the Igusa tower is not irreducible, but rather, Hida's result describes the (many) irreducible components of the pullback of the Igusa tower over any connected component of the ordinary locus and can be adapted to do the same for the $\mu$-ordinary locus.  We follow \cite{hidaIRR}.

Fix a connected component of the $\mu$-ordinary.  By abuse of notation, we still denote it by $\cS$, and the pullback of the Igusa tower by $\Ig_\mu$.

For each $n,m\geq 1$, we define 
\[\det:\Igmunm\to \underline{\Isom}_\CD(\wedge^{\rm top} \X[p^n],\wedge^{\rm top} \CA[p^n]/\cS_m)\cong \left(\CO_{B,p}/p^n\CO_{B,p}\right)^\times,\]
where the latter isomorphism follows from the fact that the sheaf $\wedge^{\rm top} \CA[p^n]/\cS_m$ is constant (see \cite[Section 3.3]{hidaIRR}, and also \cite{miaofen}, which gives a notion of a top exterior power for Barsotti--Tate groups corresponding to the top exterior power of the associated modules).
We define $\Ig_\mu^{SU}$ to be the pullback of $1\in\CO_{B,p}^\times $.

\begin{defi}(\cite[Definition 1.1]{Zong}) 
A $\CD$-enriched abelian variety $A$ over an algebraically closed field $\F$ of characteristic $p$ is hypersymmetric if 
\[ {\rm End}^0_\CD (A)\otimes_\Q\Q_p =  {\rm End}_\CD \left(H^1_{dR}(A)[\frac{1}{p}]\right).\]
A point $x$ of $\sh$ is called hypersymmetric if $\CA_x$ is hypersymmetric.
\end{defi}

\begin{prop}\label{irr}
Maintain the above notation.  Assume there exists a hypersymmetric point which is $\mu$-ordinary.

Then $\Ig_\mu^{SU}\rightarrow \cS$ is a geometrically irreducible component of $\Ig_\mu$.
\end{prop}

\begin{proof}
In \cite[Definition 4.20, and Theorem 3.1]{hidaIRR}, Hida proves the irreducibility of the ordinary Igusa tower over unitary Shimura varieties. The argument given in {\it loc. cit.} relies on the existence of ordinary hypersymmetric points (\cite[Section 3.5]{hidaIRR}), and it applies as is to our setting, with the role of ordinary hypersymmetric points replaced by hypersymmetric points on the $\mu$-ordinary locus.  
\end{proof}

\begin{rmk}
In \cite[Theorem 5.1]{Zong}, Zong gives necessary and sufficient conditions for the existence of hypersymmetric points on (each connected component of) Newton polygon strata of PEL-type Shimura varieties.  In \cite{Xiao}, Xiao checks that these conditions are satisfied by the $\mu$-ordinary stratum of unitary Shimura varieties when the degree of the primes $v$ above $p$ in $E$ is constant (e.g., for $p$ inert in $E$).  Note that \cite[Proposition 2.3.12]{moonen} implies the existence of hypersymmetric points when $p$ is inert in $F$.
\end{rmk}

\begin{rmk}\label{trans}
Recall our assumption that the subgroup $\torus(\Z_p)\subseteq\levi(\Z_p)$ acts transitively on the set of connected components of $\Ig_\mu$.
It follows from Proposition \ref{irr} that the connected components of $\Ig_\mu$ are precisely the fibers of the morphism $\det:\Ig_\mu\to \CO_{B,p}^\times$.
\end{rmk}
In the following, for simplicity, we write $\Ig:=\Igmu$.

\section{$p$-adic Automorphic forms and Congruences in the $\mu$-ordinary setting}\label{congruences-section}

The goal of this section is to explore to what extent we can realize classical and $p$-adic automorphic forms as global functions over the $\mu$-ordinary Igusa tower.

\subsection{$p$-adic automorphic forms over the $\mu$-ordinary Igusa tower}
Similarly to \cite[Section 8.1.1]{Hida}, which addresses the ordinary setting, we define the space of $p$-adic global functions on the Igusa tower 
$$V:=\varprojlim_m\varinjlim_n V_{n,m}$$ where for each $n,m \geq 1$, $V_{n,m}:=H^0\left(\Ig_{n,m}, \CO_{\Ig_{n,m}}\right).$
The natural right action of $\levi\left(\ZZ_p\right)$ on the Igusa tower defines a left action on $\Vmu$.

 \begin{defi}\label{omoit-defn}
We define the space of {\it $p$-adic automorphic forms over the $\mu$-ordinary Igusa tower} (abbreviated to {\it $p$-adic automorphic forms OMOIT}) to be $$V^{\nilp (\ZZ_p)}\subset V.$$ (Recall $\nilp$ is the unipotent radical of our choice of a Borel subgroup $\borel$ of of $\levi$.) 
\end{defi}

In the following, we simply write $\Vmu:=V^{\nilp (\ZZ_p)}$. 

\begin{rmk}
Note that when the ordinary locus is nonempty, Definition \ref{omoit-defn} agrees with Hida's definition in \cite{Hida} of the space of $p$-adic automorphic forms.
\end{rmk}

\subsection{$p$-adic forms over the $\mu$-ordinary locus}\label{UO}

We maintain the notation of Section \ref{autsheaf}.

Over the formal $\mu$-ordinary locus $\cS$ over $\W$, we write $\omega_\bullet$ for the sheaf $\omega:=\omega (\CA)$, for $\CA$ the universal abelian scheme over $\cS$, endowed with the filtration induced by the slope filtration of $\CA[p^\infty]$. We define the locally free sheaf \[\uo:=\gr(\omega_\bullet).\]   
The $\CD$-structures on $\CA$ induce a canonical decomposition $\uo=\oplus_{\tau\in\T} \uo_\tau^{\oplus r}$, where for each $\tau\in\T$, \[\uo_\tau=\oplus_{t=0}^{s_\tau} \gr^t(\omega_\tau).\]

\begin{rmk} It follows from the explicit description of the Dieudonn\'e crystal of the universal deformation of a $\mu$-ordinary \BT module over $\F$ (\cite[Section 2.1.7, and Proposition 2.1.9]{moonen}) that, for each $\tau\in\T$ and $t=0,\dots, s_\tau$, 
the sheaves $\gr^t(\omega_\tau)$ are locally free of rank $m^\tau_t$ for $t\geq i_\tau$ and vanish otherwise (notations as in Remark \ref{details}; see also the proof of Proposition \ref{localsplit}).
\end{rmk}

In the following, we adapt the classical construction of automorphic sheaves (as in Section \ref{autweight}) to our context, with $\uo$ in place of $\omega$ and $\levi$ in place of $\Levi$.

For each dominant weight ${\kappa}$  of $\levi$, we construct the sheaves $\uo^{\kappa}$ 
$$\uo^{\kappa}:={\mathbb S}_{\kappa} (\uo)$$
over $\cS$.
Alternatively,  we define
\begin{align*}
\Emu:=  \underline{\Isom}_{\CO_{B, (p)}\otimes \CO_\cS}\left(\uo, \CO^{g}_{\cS}\right):=
\prod_{\tau\in \T} \left(\oplus_{t=i_\tau}^{s_\tau} \underline{\Isom}\left(\gr^t(\omega_\tau), \CO^{m^\tau_t}_{\cS, \tau}\right)\right),
\end{align*}
and for any algebraic representation $\left(\varrho, M_\varrho\right)$ of $\levi$, we construct the sheaves 
\begin{align*}
\CE_\mu^\varrho:=\CE_\mu\times^{\levi} M_\varrho
\end{align*}
over $\cS$.
As before, we note that for any dominant weight $\kappa$ of $\levi$ and $\varrho=\varrho_\kappa$ our choice of $\ZZ_p$-lattice of the irreducible representation of $\levi$ of highest weight $\kappa$,  the sheaves $\uo^\kappa$ and $\CE_\mu^\varrho$,are canonically identified.  

\begin{defi}\label{OMOL-defn}
We call the sections of $\CE_\mu^\varrho$ {\it $p$-adic forms (of weight $\varrho$) over the $\mu$-ordinary locus}, or {\it $p$-adic forms (of weight $\varrho$) OMOL}.  Going forward, when the meaning is clear from context, we sometimes drop ``OMOL'' and just say ``$p$-adic automorphic form.''
\end{defi}

\begin{rmk}
When the ordinary locus is nonempty, the $p$-adic forms defined in Definition \ref{OMOL-defn} are the same as the automorphic forms in the vector bundle over the ordinary locus considered in \cite{EFMV} and by Hida in \cite{Hida} (see \cite[Remark 2.4.1]{EFMV}). 
\end{rmk}

\begin{rmk}
We now explain the names {\it OMOL} and {\it OMOIT}.  While the naive approach might be to call our forms {\it $\mu$-ordinary} $p$-adic automorphic forms, that name seems to imply a strong connection with a projector analogous to Hida's ordinary project $e$ formed from powers of the $U_p$-operator.  While such operators will play an important role in our subsequent work building on the present paper, they are not part of this paper.  Simply referring to the space $V^{\nilp (\ZZ_p)}$ as the space of {\it $p$-adic automorphic forms} is not precise enough.  Indeed, it immediately leads to the question of which space of $p$-adic forms we are considering (e.g., the approach of Serre?  Katz?  Hida?).  While our approach builds on Hida's approach, calling it Hida's $p$-adic automorphic forms would imply we might consider an empty set (the ordinary locus), whereas our space is always nonempty.  Thus, we add the abbreviation OMOIT to be clear about the new space we have constructed and note that in the special case in which the ordinary locus is nonempty, we recover Hida's $p$-adic automorphic forms.
\end{rmk}

It follows from the construction of the Igusa tower that for any integers $n,m\geq 1$, with $n\geq m$, the universal Igusa level structure on $\Ig_{n,m}$ induces an $\CO_{B,(p)}$-linear isomorphism
$$\uo_{\Ig_{n,m}}\cong \CO_{\Ig_{n,m}}\otimes_{\W} \omega_{\X^\can},$$
where $\uo_\Ig$ denotes the pullback of $\uo$ to the Igusa tower, and  $\omega_{\X^\can}$ is the module of invariant differential of the $\CD$ enriched \BT group $\X^\can=\X^\can(\CD)$ over $\W$. Given the canonical decomposition $\uo=\oplus_{\tau\in\T}\left(\oplus_{t=0}^{s_\tau}\gr^t(\omega_\tau)\right)$, such an isomorphism is equivalent to the collection of trivializations over $\Ig_{n,m}$
$$\gr^t(\omega_\tau)_{\Ig_{n,m}}\cong \CO_{\Ig_{n,m}}\otimes_{\W} \omega_{\X^\can(\lambda_t),\tau}^{\oplus m^\tau_t},$$ 
for $\tau\in\T$ and $t=i_\tau, \dots , s_\tau$.

Following Hida's theory,  starting from the above trivialization of $\uo$ over the Igusa tower, for all dominant weights $\kappa$ of $\levi$, we construct by Schur functors canonical trivializations of the pullbacks of $\uo^{\kappa}$ over $\Ig_{m,m}$, for all $m\geq1$. Such trivializations, composed with the $\nilp$-equivariant functional $\ell^\kappa:\varrho_\kappa\to \W$ introduced in Section \ref{autweight}, define a morphism on global sections
$$\Psi_{\kappa}:H^0\left(\cS, \uo^{\kappa}\right)\rightarrow \Vmu[\kappa]\subset \Vmu.$$

We define  
 $\Psi:=\oplus_{\kappa}\Psi_{\kappa}$, where $\kappa$ varies among all dominant weights of $\levi$,
$$\Psi:\oplus_{\kappa} H^0\left(\cS, \uo^{\kappa}\right)\rightarrow \Vmu .$$ 

\begin{prop} \label{psi}
Maintain the above notation.
\begin{enumerate} 
\item{For each dominant weight $\kappa$ of $\levi$, 
the map $\Psi_{\kappa}$ is injective.}
\item{The map $\Psi$ is injective and its image is $p$-adically dense in $\Vmu$.}
\end{enumerate}
\end{prop}

\begin{proof}
The proof is similar to \cite[Theorem 8.3]{Hida}.
\end{proof}

\subsection{Realizing $p$-adic automorphic forms as $p$-adic forms OMOIT}\label{realization-section}

We compare the $p$-adic automorphic forms OMOL we constructed above with (classical) $p$-adic automorphic forms.

\begin{prop}\label{grade}
The notation remains the same as directly above and the same as in Section \ref{LRrule}.
Let $\kappa$ be a dominant weight of $\Levi$.\begin{enumerate}

\item Each 
$\uni$-stable filtration of $\rho_\kappa\vert_{\para}$ induces a filtration on $\omega^\kappa$.
 The sheaf $\gr((\omega^\kappa)_\bullet)$ is independent of the choice of filtration on $\rho_\kappa\vert_{\para}$. 
\item There is a canonical morphism
$$\gr((\omega^\kappa)_\bullet)\to \bigoplus_{\kappa'\in\M_\kappa} \uo^{\kappa'},$$
which is an isomorphism if $p$ is sufficiently large, or after tensoring with $\Q_p$.
\item 
There is a canonical projection $\pi^\kappa: \omega^\kappa\twoheadrightarrow \uo^\kappa.$

\end{enumerate}
\end{prop}
\begin{proof}
Let $\omega_\bullet$ denote the slope filtration on $\cS$, and define  over $\cS$
$$\Pmu:=  \underline{\Isom}_{\CO_{B, (p)}\otimes \CO_\cS}\left(\omega_\bullet, (\CO_{\cS})^{g}_\bullet\right):=
\bigoplus_{\tau\in \T}\underline{\Isom}\left((\omega_\tau)_\bullet, (\CO_{\cS, \tau}^{\oplus \cf(\tau)})_\bullet\right),$$
where the filtration on  $\CO_{\cS, \tau}^{\oplus \cf(\tau)}$ is induced by the ordered partition $\{m_{s_\tau}^\tau, \dots, m_{i_\tau}^\tau\}$ of $\cf(\tau)$. 
Note that by definition $\Pmu\subseteq \CE_{\vert\cS}$, and we have a canonical projection $\Pmu\twoheadrightarrow\Emu$. 

The inclusion $\Pmu\subseteq \CE_{\vert\cS}$ implies that, for all representations $(\rho, M_\rho)$ of $\Levi$, we have identifications of sheaves over $\CS$
$$\CE^\rho_{\vert\cS}=(\CE\times^\Levi M_\rho)_{\vert\cS}= {\mathcal P}_\mu \times^{\para} M_\rho.$$
In particular, for all dominant weights $\kappa$ of $\Levi$, each $\uni$-stable filtration of $\rho_\kappa\vert_{\para}$ induces a filtration on the pullback over $\cS$ of the sheaf $\omega^\kappa$. 
In particular, 
the natural projection $\rho_\kappa\vert_{\para}\twoheadrightarrow \varrho_\kappa$ induces a map on sheaves $\omega^\kappa\to \uo^\kappa.$

The projection $\Pmu\twoheadrightarrow\Emu$ implies that, for all representations $(\rho, M_\rho)$ of $\Levi$, we have identifications of sheaves over $\CS$
$$\gr({\mathcal P}_\mu \times^{\para} M_\rho)=\Emu\times^{\levi} M_\rho$$
(recall $\gr(\rho_\kappa\vert_{\para})= \rho_\kappa\vert_{\levi}$).
In particular, for $\rho=\rho_\kappa$, the equality $\rho_\kappa\vert_{\levi}=\oplus_{\kappa' \in\M_\kappa}\varrho_{\kappa'}$ implies that there exists a canonical morphism
$$\gr(\omega^\kappa)\to\oplus_{\kappa'\in\M_\kappa} \uo^{\kappa'},$$
which is an isomorphism if $p$ is sufficiently large, or after tensoring with $\Q_p$ (see Remark \ref{lattices}).
\end{proof}

For each 
weight $\kappa$ of $\Levi$, we define $\varPhi_\kappa$ as the composition of $H^0(\cS,\pi^\kappa)$ with 
$\Psi$, 
\[\varPhi_\kappa: H^0(\cS,\omega^\kappa)\to H^0(\cS,\uo^\kappa)\to V[\kappa]\subset V,\]
and write $\varPhi:=\oplus_\kappa\varPhi_\kappa$. The map $\varPhi$ realizes $p$-adic (and thus also classical) automorphic forms as $p$-adic forms OMOIT. 

For scalar weights $\kappa$, $\varPhi_\kappa$ is injective.
Unfortunately, for non-scalar weights $\kappa$, $\varPhi_\kappa$ is not injective.  Also, the image of $\varPhi$ is not $p$-adically dense in $V$ (because dominant weights for $\levi$ need not be dominant for $\Levi$).

\begin{rmk}  
The $\mu$-ordinary Hasse invariant $\hasse$ (as defined in Section \ref{HasseGN}) satisfies
\[\varPhi(\hasse)\equiv 1\mod p.\]
Moreover, for each scalar weight $\kappa$ of $\Levi$, the canonical trivialization over the Igusa tower $\omega^\kappa\cong \uo^\kappa\cong \CO_{\Ig}$ agrees modulo $p$ with (the pullback of) the identification $|\omega|^{m_0}=\CO_{\bS}$ over $\bS$.
\end{rmk}

\subsubsection{Local realizations} The connection between $p$-adic automorphic forms and $p$-adic forms OMOIT is stronger when working locally.

\begin{prop} \label{localsplit}
The notation remains the same as above. Let $x_0\in\cS(\F)$, and $S^\wedge_{x_0}$ denote the formal completion of $S$ at $x_0$. 
\begin{enumerate}
\item The filtration of $\omega$ is canonically split over $S^\wedge_{x_0}$.  That is, we have a canonical isomorphism over $\CO^\wedge_{\cS,x_0}$ $$\omega_{x_0}\cong \uo_{x_0}.$$ 

\item For each dominant weight $\kappa$ of $\Levi$, there is a canonical morphism over $\CO^\wedge_{\cS,x_0}$
$$\omega^\kappa_{x_0}\to \oplus_{\kappa'\in\M_\kappa} \uo^{\kappa'}_{x_0},$$
which is an isomorphism if $p$ is sufficiently large, or after tensoring with $\Q_p$.
\end{enumerate}
\end{prop}

\begin{proof}
We deduce the existence of the canonical splitting of the filtration on $\omega_{x_0}$ from the description of the Dieudonn\'e crystal of the universal deformaion of a $\mu$-ordinary \BT group in \cite[Section 2.1.7]{moonen}.  In the following, all sheaves are restricted to the formal neighborhood $\cS^\wedge_{x_0}$, but for simplicity still denoted by the same notation. 

Given the decomposition of the filtered Dieudonn\'e crystal of $A$ into subcrystals 
$$\omega(A)=\bigoplus_{\co\in\FO} (\omega_\co)^{\oplus r}\subset 
H^1_{dR}(A)=\bigoplus_{\co\in\FO} (M_\co)^{\oplus r},$$
where $\omega_\co:= \oplus_{\tau\in\co} \omega_\tau$ and $
M_\co=\oplus_{\tau\in\co} M_\tau$, it is enough to prove that for each orbit $\co$ the filtration of $\omega_\co$ is canonically split.

Fix $\co$, and write $M=M_\co$, $\omega=\omega_\co$, $\omega\subset M$. With notation as in Remark \ref{details}, let $F_1>\cdots > F_{s}$ denote the distinct values of $\cf(\tau)$ in the interval $[1,n-1]$ ($s=s_\co\geq 0$), and write $F_0:=n$ and $F_{s+1}:=0$. Then the crystal $M$ has exactly $s+1$ distict slopes. Write $M_\bullet$ for its slope filtration,
\[0=M_0\subset M_1\subset \cdots \subset M=M_{s+1},\]
and $M^i:=M_i/M_{i-1}$, $i\geq 1$.
In \cite[Section 2.1.7, and Propositions 2.1.8 and 2.1.9]{moonen}, Moonen gives an explicit description of the pair $(M, \omega)$  over $\cS^\wedge_{x_0}$ (in {\it loc. cit.} $M={\mathcal M}$ and $\omega={\rm Fil}^1({\mathcal M})$). In particular, $M=\oplus_1^{s+1} M^i$, and for all $j=1,\ldots, s+1$, $M_j=\oplus_1^jM^i$.

For each $i=0, \ldots, s+1$, define \[\co_i:=\{\tau\in\co|\cf(\tau)=F_i\},\] and $\co_{< j}:=\cup_{i< j}\co_i$. Also, for each $\tau\in\co$, we write $\iota(\tau):=i$ if $\tau\in \co_i$.
Then it follows from the definitions that, over $\cS^\wedge_{x_0}$, for all $\tau$ such that $\iota(\tau)<s+1$
$$\omega_\tau=\oplus_{\iota(\tau)+1}^{s+1} M^{i}_{\tau},$$
and $0$ otherwise.
Moreover, the filtration $\omega_\bullet$ of $\omega$ induced by the slope filtration of $M$ satisfies, for each $\tau\in\co$,
$\omega_{j,\tau}=\oplus_{\iota(\tau)+1}^{j} M^{i}_{\tau}$ for $j> \iota(\tau)$, and $0$ otherwise.

In particular, for each  $j=1, \ldots, s+1$, \[\omega_{j}=\omega_{j-1}\oplus (\oplus_{\tau\in \co_{< j}} M_\tau^j).\] 
\end{proof}

Fix $x_0\in \cS(\F)$, and let $\kappa$ be a dominant weight of $\Levi$. For each $x\in \Ig(\F)$ above $x_0$,  and $\kappa'\in\M_\kappa$, we define
\[\Phi^{\kappa,\kappa'}_{x}: H^0(\cS,\omega^\kappa)\to \omega^\kappa_{x_0}\to \oplus_{\kappa' \in \M_\kappa}\uo_{x_0}^{\kappa'}\to \uo_{x_0}^{\kappa'}\to \CO^\wedge_{\Ig,x},\]
as the composition of localization at $x_0$, with the canonical morphism of part (2) of Proposition \ref{localsplit}, followed by $\Psi_{\kappa', x}$, the localization of $\Psi_{\kappa'}$ at $x$.

Note that for each $g\in\levi(\ZZ_p)$, $\Psi_{\kappa',x^g}= \Psi_{\kappa',x}\circ g$. In particular, it follows from Proposition \ref{psi} that the map \[\Phi^\kappa_{x_0}:=\prod_{x,\kappa'}\Phi^{\kappa,\kappa'}_{x}:H^0(\cS,\omega^\kappa)\to\prod_{x,\kappa'} \CO^\wedge_{\Ig,x}\] is injective if $\cS$ is connected.  In the following, we also write $\Phi_x:=\Phi^\kappa_x:=\sum_{\kappa'\in\M_\kappa} \Phi^{\kappa, \kappa'}_{x}:H^0(\cS,\omega^\kappa)\to\CO^\wedge_{\Ig,x}$.

\begin{remark}

For each pure weight $\kappa$ of $\Levi$, the morphism $\Phi^{\kappa,\kappa}_x$ agrees with the composition of $\varPhi_\kappa$ with the localization at $x$,
\[ H^0(\cS,\omega^\kappa)\to H^0(\cS,\uo^\kappa)\to V\to\CO^\wedge_{\Ig,x}.\]
\end{remark}

\subsection{$p$-adic $u$-expansion principle and congruences}\label{cong-section}
In this section, we generalize the results of \cite[Section 5]{CEFMV} to the $\mu$-ordinary setting. We refer to {\it loc. cit.} for more details.
Here, we work under the assumption that $\torus(\Z_p)$ acts transitively on the connected components of $\Ig$.
As in Hida's work, the restriction of the Igusa tower over a connected component of the $\mu$-ordinary locus is not irreducible.
As stated, the $p$-adic $\uu$-expansion principle (like its analogue, the $q$-expansion principle) relies on the transitivity of the action of $\torus(\Z_p)$ on the set of connected components of $\Ig$. 

We choose the notation $\uu$, instead of $t$ as in \cite{CEFMV}, for the coordinate in local expansions at $\mu$-ordinary CM points, because it agrees with Moonen's conventions in \cite{moonen}, which play an important role in some of the notationally heavy portions of this paper. 

\newcommand{\loc}{{\rm loc}}
Following Hida, to establish an analogue of the $p$-adic $q$-expansion principle, we fix a connected component $\cS_0$ of the $\mu$-ordinary locus $\cS$, together with a marked point $x_0$ on $\cS_0$, and replace $p$-adic automorphic forms on $\cS$ with their restriction to $\cS_0$, and the Igusa tower by its pullback to $\cS_0$.  Alternatively, one can work with many marked points on $\cS$ at once, one point on each connected component.

\subsubsection{Canonical parameters at $\mu$-ordinary points}
Fix a point $x$ of $\Ig(\W)$, above a $\mu$-ordinary point $x_0$. In \cite[Section 2.1.7]{moonen}, Moonen defines a set of local parameters $\uu$ of $\cS_0$ at $x_0$, associated with a trivialization of the fiber at $x_0$ of the universal $\CD$-enriched \BT group. With our notation, this is equivalent to the choice of a point $x$ of the Igusa tower lying above $x_0$.  

In the following, we denote the choice of parameters $\uu$ at the point $x_0$ associated with the point $x\in\Ig(\W)$ as 
\[\beta^*_x:\CO^\wedge_{\cS,x_0}\cong \CO_{\Ig,x}^\wedge\cong  \W[\![\uu]\!]:=\W[\![u^\tau_{r,s}|\tau\in\T, r,s=1,\dots n]\!]\]
where by definition, for each $\tau\in \T$, $u^\tau_{r,s}:=0 $   if either $r \leq n-\cf(\tau)$ or $s > n-\cf(\tau)$ (in {\it loc. cit.} $i=\tau$, and $d=n$).
We write $\loc_x:V\to \CO_{\Ig,x}^\wedge$ for the localization at $x$.

\begin{rmk}
The results in this section do not rely on the special properties of the parameters $\uu$. In fact, they could as easily be stated in terms of the localization at $x$. We choose to state them in terms of the associated power series in $\W[\![\uu]\!]$ to stress the analogue with the $p$-adic $q$-expansion principle, in the ordinary case.  
\end{rmk}

\begin{defi}
For any global function $f\in V$ on the Igusa tower, we define the $\uu$-expansion of $f$ at $x$  as
\[f(\uu)=f_x(\uu):=\beta^*_x(\loc_x(f))\in\W[\![\uu]\!].\]

For each $p$-adic form $f\in H^0\left(\cS, \uo^{\kappa}\right)$ of weight $\kappa$, $\kappa$ a dominant weight of $\levi$, we set
$$f(\uu):=\Psi(f)(\uu)\in\W[\![\uu]\!].$$

For each classical (respectively, $p$-adic) automorphic form $f\in H^0\left(\Sh, \omega^{\kappa}\right)$ 
(respectively, $f\in H^0\left(\cS, \omega^{\kappa}\right)$) of weight $\kappa$, for $\kappa$ a dominant weight of $\Levi$, we set 
$$f(\uu):=\varPhi(f)(\uu)\in\W[\![\uu]\!].$$
\end{defi}

\begin{prop}\label{prop-st}Maintain the above notation.
\begin{enumerate}

\item For any $f\in \Vmu$, $f=0$ if and only if $(g\cdot f)(\uu)=0$ for all $g\in\torus (\Z_p)$.
\item For any dominant weight $\kappa$ of $\levi$ and $f\in\Vmu[\kappa]$, $f=0$ if and only if $f(\uu)=0$.
\item For $m\geq 1$, $\kappa_i$ dominant weights of $\levi$, and $f_i\in\Vmu[\kappa_i]$, $i=1,2$, $f_1\equiv f_2\mod p^m$ if and only if for all $g\in \torus(\Z_p)$ $$\kappa_1(g) f_{1}(\uu)\equiv \kappa_2(g) f_{2}(\uu) \mod p^m.$$
\end{enumerate}
\end{prop}
\begin{proof}The statements follow immediately from the transitivity of the action of $\torus(\Z_p)$ on the set of connected components of $\Ig$ (Remark \ref{trans})  and the equalities, for $g\in\torus(\Z_p)$, $\loc_x(g\cdot f)=\loc_{x^g}(f)$ for $f\in \Vmu$, and $\loc_x(g\cdot f)=\kappa(g)\loc_x(f)$ for $f\in\Vmu[\kappa]$. 
(The arguments of \cite{CEFMV} still apply in our setting.)
 \end{proof}

The next corollary is an immediate consequence of Proposition \ref{prop-st} combined with Proposition \ref{psi}.

\begin{cor}Maintain the above notation.
\begin{enumerate}
\item For any dominant weight $\kappa$ of $\levi$ and $f\in H^0\left(\cS, \uo^{\kappa'}\right)$, $f=0$ if and only if $f(\uu)=0$.
\item For $m\geq 1$, $\kappa_i$ dominant weights of $\levi$, and $f_i\in H^0\left(\cS, \uo^{\kappa_i}\right)$, $i=1,2$, 
$f_1\equiv f_2\mod p^m$ if and only if for all $g\in \torus(\Z_p)$ $$\kappa_1(g) f_{1}(\uu)\equiv \kappa_2(g) f_{2}(\uu) \mod p^m.$$
\end{enumerate}
\end{cor}

\begin{cor}\label{necc}
Let $m\geq 1$.  Let $f_i$ be classical or $p$-adic automorphic forms of scalar weight $\kappa_i$, $i=1,2$.
Then $f_1\equiv f_2\mod p^m$  if and only if
$$\kappa_1(g)\equiv\kappa_2(g)\mod p^m ,\text{ for all } g\in \torus(\Z_p)$$ and 
$$f_1(\uu)\equiv f_{2}(\uu) \mod p^m.$$
\end{cor}

\begin{rmk}
For non-scalar weights,  the above condition is necessary but not sufficient.
\end{rmk}

Finally, to state a sufficient condition for general weights, for any $p$-adic (or classical) automorphic form $f$ of weight $\kappa$,
 we consider the $\uu$-expansions 
\[f_x^{(\kappa')}(\uu):=\beta^*_x(\Phi_x^{\kappa,\kappa'}(f))\in\W[\![\uu]\!]\] 
defined for all $\kappa'\in\M_\kappa$
and all points $x\in\Ig(\W)$ lying above a fixed $\mu$-ordinary point $x_0$.
(Note that $f(\uu)=f_x^{(\kappa)}(\uu).$)

\begin{cor}\label{cong-mult-cor}
Maintain the above notation. Fix $x_0\in\cS_0(\W)$.
\begin{enumerate}
\item For any dominant weight $\kappa$ of $\Levi$, and $f$ a classical or $p$-adic automorphic form of weight $\kappa$, $f=0$ on $\cS_0$ if and only if 
for all $x\in\Ig(\W)$ above $x_0$, $$f_x^{(\kappa')}(\uu)=0 \text{ for all }\kappa'\in\M_\kappa.$$

\item For $m\geq 1$, $\kappa_i$ dominant weights of $\Levi$, and $f_i$ classical or $p$-adic automorphic forms, respectively, of weight $\kappa_i$, $i=1,2$, we have 
$f_1\equiv f_2\mod p^m$  if
$$\kappa_1(g)\equiv\kappa_2(g)\mod p^m \text{ for all } g\in \torus(\Z_p)$$ and for all pairs $\kappa'_i\in\M_{\kappa_i}$, $i=1,2$, 
with $\kappa'_1(g)\equiv\kappa'_2(g)\mod p^m $, for all $x\in\Ig(\W)$ above $x_0$
$$f_{1,x}^{(\kappa_1')}(\uu)\equiv f_{2,x}^{(\kappa_2')}(\uu) \mod p^m.$$\label{cor-item-2}
\end{enumerate}
\end{cor}

\begin{rmk}
The congruence condition given in Part \eqref{cor-item-2} is both necessary and sufficient if, furthermore, we assume  that, for each $i=1,2$, the weights $\kappa'\in\M_{\kappa_i}$ are all distinct modulo $p^m$.  
\end{rmk}

\section{Structure theorems in the $\mu$-ordinary case}\label{structure-mo}
This section develops structural results concerning the Gauss--Manin connection, the Kodaira--Spencer morphism, and a canonical complement to $\omega$ in the $\mu$-ordinary setting, as needed in Section \ref{do-mo} to construct differential operators, which we use in the subsequent sections to construct new $p$-adic automorphic forms and families of $p$-adic automorphic forms.

{For simplicity, we assume $B=F$. By the assumption that the prime $p$ is unramified in $B$, the general case follows from this special case by Morita equivalence.}

  \subsection{Standard constructions}  This section recalls the definitions of the Gauss--Manin connection and the Kodaira--Spencer morphism.
Throughout this section, we denote by $S$ a smooth scheme over a scheme $T$ and by $\pi: X\rightarrow S$ a proper morphism of schemes.  For any such schemes, we denote by $\Omega_{X/S}^\bullet$ the complex $\wedge^\bullet \Omega_{X/S}^1$ on $X$ whose differentials are induced by the canonical K\"ahler differential $\CO_{X/S}\rightarrow \Omega^1{X/S}$.  The de Rham complex $\left(\Omega_{X/S}^\bullet, d\right)$ admits a canonical filtration
\begin{align}\label{filtrationGM}
\Fil^i\left(\Omega_{X/T}^\bullet\right):=\Image\left(\pi^*\Omega_{S/T}^i\otimes_{\CO_X}\Omega_{X/T}^{\bullet-i}\rightarrow\Omega_{X/T}^\bullet\right).
\end{align}
For $\pi$ smooth, the sequence
\begin{align*}
0\rightarrow\pi^*\Omega_{S/T}^1\rightarrow\Omega_{S/T}^1\rightarrow\Omega_{X/S}^1\rightarrow 0
\end{align*}
is exact, and the associated graded objects of the canonical filtration \eqref{filtrationGM} are
\begin{align*}
\Gr^i:=\Gr^i\left(\Omega_{X/T}^\bullet\right)\cong\pi^*\Omega_{S/T}^i\otimes_{\CO_X}\Omega_{X/S}^{\bullet-i}.
\end{align*}

\subsubsection{Gauss--Manin connection}
The Gauss--Manin connection
\begin{align*}
\nabla: \hdr^q\left(X/S\right)\rightarrow\hdr^q\left(X/S\right)\otimes_{\CO_S}\Omega_{S/T}^1
\end{align*}
is the map
\begin{align*}
d_1^{0, q}: E_1^{0, q}\rightarrow E_1^{1, q},
\end{align*}
where 
\begin{align*}
E_1^{p, q}=\mathbb{R}^{p+q}\pi_*\left(\Gr^p\right)\cong\Omega_{S/T}^p\otimes_{\CO_S}\hdr^q\left(X/S\right)
\end{align*}
is the first page of the spectral sequence ($E_r^{p, q}$, which converges to $\mathbb{R}^q\pi_*\left(\Omega_{X/T}^\bullet\right)$) obtained from the filtration \eqref{filtrationGM}.  We are interested in the case $q=1$, i.e.
\begin{align*}
\nabla: \hdr^1\left(X/S\right)\rightarrow\hdr^1\left(X/S\right)\otimes_{\CO_S}\Omega_{S/T}^1
\end{align*}

\subsubsection{Kodaira--Spencer morphism}\label{KS-class}

We briefly review the Kodaira--Spencer morphism here.  Details are available in \cite{CF, lan, EDiffOps, E09, EFMV}.  Let $\pi: A\rightarrow S$ be a smooth proper morphism of schemes (with $S$ still as above), and suppose $A$ is an abelian scheme with polarization $\lambda: A\rightarrow A\dual$.  For any such $A$, we define
\begin{align*}
\omega_{A/S}:=\pi_*\Omega_{A/S}^1.
\end{align*}
By taking the first hypercohomology of the exact sequence
\begin{align*}
0\rightarrow\Omega_{A/S}^{\bullet\geq 1}\rightarrow\Omega_{A/S}^\bullet\rightarrow\CO_A\rightarrow0,
\end{align*}
we obtain an exact sequence
\begin{align*}
0\rightarrow\omega_{A/S}\xrightarrow{\iota_{\omega_{A/S}}}\hdr^1\left(A/S\right)\xrightarrow{p_{A\dual/S}}\omega\dual_{A\dual/S}\rightarrow 0,
\end{align*}
with $\iota_{\omega_{A/S}}$ denoting inclusion.
The Kodaira--Spencer morphism $KS$ is the composition of morphisms
\begin{align*}
 \xymatrix{ 
 \hdr^1\left(A/S\right) \otimes\omega_{A\dual/S}\ar[rr]^{\nabla\otimes\iota_{A\dual/S}\qquad} &&
\left(\hdr^1\left(A/S\right)\otimes\Omega^1_{S/T}\right)\otimes\omega_{A\dual/S} \ar@{->>}[rr]^{\quad \left(p_{A\dual/S}\otimes\id\right)\otimes_{\omega_{A\dual/S}}}   &&
\omega_{A\dual/S}\dual\otimes \Omega^1_{S/T} \otimes\omega_{A\dual/S} \ar@{->>}[d]^{} \\
\omega_{A/S}\otimes\omega_{A\dual/S}\ar@{^{(}->}[u]^{\iota_{\omega_{A/S}}\otimes\id} \ar@{-->>}[rrrr]^{KS} & &&&
\Omega^1_{S/T} }
\end{align*}
with the vertical surjection denoting the canonical pairing
\begin{align*}
\omega_{A\dual/S}\dual\otimes\omega_{A\dual/S}\rightarrow\CO_S
\end{align*}
tensored with the identity map on $\Omega_{S/T}^1$.  Identifying $\omega_{A/S}$ with $\omega_{A\dual/S}$ via the polarization $\lambda: A\rightarrow A\dual$, we also view $KS$ as a morphism
\begin{align*}
KS: \omega_{A/S}\otimes_{\CO_S}\omega_{A/S}\twoheadrightarrow \Omega_{S/T}^1.
\end{align*}
The action of $\Ocmfield$ on $A$ induces a decomposition
\begin{align*}
\omega_{A/S} = \bigoplus_{\tau\in\T}\omega_{A/S, \tau}.
\end{align*}
 By \cite[Proposition 2.3.5.2]{lan}, $\KS$ induces an isomorphism
\begin{align*}
\ks: \omega^2_{A/S}\isomto\Omega_{S/T}^1,
\end{align*}
where
\begin{align*}
\omega^2_{A/S}&:=(\omega_{A/S}\otimes_{\CO_F\otimes_{\CO_T}\CO_S}\omega_{A^\vee/S})^{\lambda-sym}\\
&:= (\omega_{A/S}\otimes_{\CO_S}\omega_{A^\vee/S})/
\langle 
\begin{matrix}\lambda(y)\otimes z-\lambda(z)\otimes y\\
b x\otimes y-x\otimes b^\vee y
\end{matrix}\vert  
\begin{matrix} x\in \omega_{A/S}\\ y,z\in\omega_{A^\vee/S}\\ b\in\CO_F\end{matrix} \rangle.
\end{align*}
In particular, 
\begin{align*}
\omega^2_{A/S}&=(\bigoplus_{\tau\in\T} \omega_{A/S,\tau}\otimes_{\CO_S}\omega_{A^\vee/S,\tau})^{\lambda-sym}\\
&:=(\bigoplus_{\tau\in\T} \omega_{A/S,\tau}\otimes_{\CO_S}\omega_{A^\vee/S,\tau})/\langle 
\lambda(y)\otimes z-\lambda(z)\otimes y | y,z\in\omega_{A^\vee/S}\rangle\\
&\simeq \bigoplus_{\tau\in\T_0} \omega_{A/S,\tau}\otimes_{\CO_S}\omega_{A^\vee/S,\tau}.
\end{align*}

\subsection{A canonical complement to $\omega$ over $\cS$}\label{ur-section}
The constructions of $p$-adic differential operators in, for example, \cite{kaCM, EDiffOps, EFMV} rely on the unit root splitting discussed in \cite{kad}.  In the more general $\mu$-ordinary case, we need to work with a complement to $\omega$ that is larger than just the unit root piece and whose existence follows from work in \cite{moonen}.  To emphasize the connection with this earlier setting, we still use the notation $U$ (even though $U$ is precisely the unit root when the ordinary locus is nonempty).

Let ${\hdr^1}_\bullet$
denote the Dieudonn\'e crystal of $\CA$ over $\cS$ and $\omega_\bullet\subset {\hdr^1}_\bullet$ its Hodge filtration, both equipped with the slope filtration.  Following \cite{kad}, we deduce the existence  over $\cS$ of a canonical splitting of the sequence
$$0\rightarrow\omega_\bullet \subseteq {\hdr}^1_\bullet \rightarrow (\hdr^1/ \omega)_\bullet\rightarrow 0.$$

\begin{prop}\label{unitroot} There exists a unique submodule $U$ of ${\hdr^1}$ such that
\begin{enumerate}
\item  $U$ is $\Phi^{\mathbf e}$-stable, where $\mathbf{e}={\rm lcm}_{\co\in \mathfrak{O}} e_\co$.
\item $U$ is $\nabla$-horizontal, i.e., $\nabla(U)\subset U\otimes\Omega^1_{\cS/\W}$.
\item $U$  is a complement to $\omega$, i.e., $\hdr^1=\omega\oplus U$.\label{ur3}
\end{enumerate}
Moreover,  the filtration $U_\bullet$ of $U$, induced by the slope filtration on  ${\hdr^1}$, satisfies
$$(\hdr^1)_j=\omega_j\oplus U_j,$$
for each slope $j$;
and it is canonically split over $S_{x_0}^\wedge$, for each point $x_0\in \cS(\F)$.
\end{prop}

\begin{proof}
We use the notation introduced in the proof of Proposition \ref{localsplit}. We construct $U=\oplus_\co U_\co$, with $U_\co\subset M_\co$ a complement to $\omega_\co$. Its uniqueness follows from the listed properties.

Fix $\co$, $e=e_\co$, and write $M=M_\co$, $\omega=\omega_\co$ and construct $U=U_\co$. Let $M_\bullet$ denote the slope filtration of $M$, then
we define $$U:=\oplus_\tau M_{\iota(\tau),\tau}.$$
The stated properties are an immediate consequence of the definition of $U$, and the properties of the slope filtration (see Proposition \ref{slopefiltration}, and \cite[Proposition 2.1.9]{moonen}).
 \end{proof}

In the following we write $\U$ to denote the graded sheaf associated with $U$, and we canonically identify $\U_{x_0}\cong U_{x_0}$ over $\CO^\wedge_{\cS,x_0}$, for each $x_0\in\cS(\F)$.

\subsection{The Gauss--Manin connection}
We extend Katz's computation of the Gauss--Manin connection over the ordinary locus to the $\mu$-ordinary case.

Consider the operator $$\nabla: \omega\subset \hdr^1\to \hdr^1\otimes \Omega^1_{\cS/\W}.$$
It preserves the slope filtration, in particular it induces an operator on the graded sheaves
$\underline{\nabla}: \uo\to\gr(\hdr^1)\otimes  \Omega^1_{\cS/\W}.$

\begin{prop}\label{horizontalbasis}\label{GMST}
Maintain the above notation.  
We denote by $\alpha$ 
the isomorphism of sheaves over the Igusa tower,  $$\alpha: \uo_{\X}\otimes_{\W}\CO_{\Ig} \to \uo,$$
induced by the universal Igusa level structure.

Then for any  $\eta\in \omega_{\X}$: 
$\underline{\nabla}(\alpha(\eta))\in \U \otimes\Omega^1_{\cS/\W}.$

Furthermore, for each $x\in \Ig(\F)$ and each $\eta\in \omega_{\X}$, via the canonical splitting $\omega_x\cong \uo_x$, we have 
$\nabla_x (\alpha(\eta))\in U_x \otimes\Omega^1_{\cS/\W, x}$.
\end{prop}

 \begin{proof}
 As in the proofs of Propositions \ref{localsplit} and \ref{unitroot}, consider the decomposition into subcrystals $\hdr^1=\oplus_\co M_\co$, fix $\co$, $e=e_\co$, and write $M=M_\co$, 
$\omega=\omega_\co$, 
and $U=U_\co$.

Note that that the statement can be checked locally, over the complete local ring $\cR$ at a point $x_0$ of $\cS_1$.  That is, without loss of generality, we may assume we are in the setting of \cite[Sections 2.1.7]{moonen} (in {\it loc. cit.} $\cR= \W[\![\underline{u}]\!]$). Furthermore, since $\nabla$ is $\W$-linear, it suffices to prove the statement for a choice of  $\W$-basis of $\omega_\X=\uo_\X$ compatible with the slope decomposition.

In \cite[Proposition 2.1.9]{moonen},  Moonen computes the matrix of 1-forms of the Gauss--Manin connection on $M$, with respect to an explicit choice of a basis of $M$.  
We observe that the chosen basis contains the image under $\alpha$ of (an explicit choice of) a $\W$-basis of $\omega_\X$.
We quickly recall Moonen's notation and results.

Let ${\mathcal B}=\{\alpha^\tau_j|\tau\in\co, j=1\dots, n\}$ denote the basis of $M$, $M=M_\X\otimes_\W \cR$, as defined in \cite[Section 1.2.3]{moonen} (in {\it loc. cit.} the element $\alpha^\tau_j$ are denoted by $e_{i,j}$, with $i=\tau$,  ${\mathfrak I}=\co$, and $d=n$).  
By definition, for each $\tau\in\co$, $\{\alpha^\tau_j|j=1,\dots, n\}$ is a basis of $M_{\X,\tau}$ over $\W$, such that, for each $j=1,\dots ,n$,
$$\Phi_\tau^{e}(\alpha^\tau_j)=p^{a(j)}\alpha^\tau_j,$$ where $a(j)=\#\{\tau|\cf(\tau)>n-j\}$. 
In particular, $\{\alpha^\tau_j\mid j>n-\cf(\tau)\}$ is a basis of $\omega_\tau\subset M_\tau$, which arises as the image under $\alpha$ of a $\W$-basis of $\omega_\X$. 
Note that, for each slope $a=a_i$, we have $\alpha^\tau_j\in M_a$ if and only if $j>n-F_i$, for all $\tau\in\co$, and $U^a:=\oplus_\tau U^a_\tau$ has basis $\{\alpha^\tau_j| j\leq n-F_i, \tau\in\co\}$.

As in \cite[Proposition 2.1.9]{moonen}, we denote  by $D^\tau=(D^\tau_{r,s})_{r,s=1,\dots ,n}$ the matrix of 1-forms of $\nabla$ with respect to the given basis $\{\alpha^\tau_j |j=1,\dots, n\}$ of $M_\tau$, $\tau \in\co$. 
To  prove our statement,   it suffices to check that, for each $\tau\in\co$, and for all $  j>n-\cf(\tau)$,  $$\nabla(\alpha^\tau_j)\equiv 0\mod U \otimes \Omega^1, $$ or equivalently, $D^\tau_{v,j}=0$ if $j>n-\cf(\tau)$ and $v>n-\cf(\tau)$.

Following {\it loc. cit.}  the inclusions $\nabla (U^a)\subset U^a\otimes \Omega^1,$ $a=0,\dots, e$, imply $D^\tau_{v,j}=0$ if  $v\leq  n-F_i$ and $j>n-F_i$, for all $\tau\in\co$, and $i=0,1,\dots ,s$.
Also, from the equality $\nabla\circ F=(f\otimes {\rm id})\circ F$, we obtain, 
for all $\tau\in\co$ and $j, v\in\{1, \dots, n\}$,
$$D_{v,j}^\tau+\sum_{r=1}^n u_{r,j}^\tau D_{v,r}^\tau+du_{v,j}^\tau=d\phi(D^{\tau^{\sigma^{-1}}}_{v,j})+\sum_{l=1}^du_{v,l}^\tau \cdot d\phi(D_{l,j}^{\tau^{\sigma^{-1}}}),$$
 where $u^\tau_{r,s}:=0$ if either $r\leq n-\cf(\tau)$ or $s>n-\cf(\tau)$. (Recall that $\phi$ on $\cR$ is defined by $\phi_{\vert\W}=\sigma$ and $\phi(u^\tau_{r,s})=(u^\tau_{r,s})^p$.)

Fix $\tau\in\co$  and assume $j>n-\cf(\tau)$ and $v>n-\cf(\tau)$. From  $j>n-\cf(\tau)$, we deduce  $u^\tau_{r,j}=0$ for all $r$, and also $D^{\tau'}_{l,j}=0$ for $l\leq n-\cf(\tau)$, for all $\tau'$.
Thus, the above equalities become
$$D_{v,j}^\tau = d\phi(D^{\tau^{\sigma^{-1}}}_{v,j})+\sum_{l>n-\cf(\tau)}^du_{v,l}^\tau \cdot d\phi(D_{l,j}^{\tau^{\sigma^{-1}}})=d\phi(D^{\tau^{\sigma^{-1}}}_{v,j}),$$
which implies the equation $D_{v,j}^\tau=d\phi^e (D^\tau_{v,j}).$ 
We deduce that $D^\tau_{v,j}\equiv 0 \mod p^m$, for all $m\geq 1$, and conclude that $D^\tau_{v,j}=0$, for all  $j>n-\cf(\tau)$ and $v>n-\cf(\tau)$.
\end{proof}

\subsection{The Kodaira--Spencer morphism}

We study the Kodaira--Spencer morphism over the $\mu$-ordinary Igusa tower.  Let $\CA$ denote the universal abelian scheme over the Igusa tower $\Ig$ over $\W$, and let $\omega=\omega_{\CA/\Ig}$.  Recall the notations from section \ref{KS-class}.
Write \begin{align*}
\omega^2&:=(\omega_{\CA/\Ig}\otimes_{\CO_F\otimes_\W \CO_{\Ig}} \omega_{\CA ^\vee/\Ig})^{\lambda-sym}
 \cong (\omega\otimes_{\CO_F\otimes_\W \CO_{\Ig}}U^\vee)^{\lambda-sym}\\
&\cong (\bigoplus_{\tau\in\T}  \omega_\tau \otimes_{\CO_\Ig}U^\vee_\tau)^{\lambda-sym}
\cong \bigoplus_{\tau\in\T_0} \omega_\tau \otimes_{\CO_\Ig}U^\vee_\tau.
\end{align*}
For each orbit $\co\in\mathfrak{O}$, write $s=s_\co=s_\tau$, for $\tau\in\co$. For each $i,j\in\{0,\dots, s\}$, and $\tau\in\co$ we define
$$\gr_\tau^{i,j}(\omega\otimes U^\vee):=\gr^i(\omega)_\tau\otimes_{\CO_{\Ig}} \gr^j(U)^\vee_\tau$$
and  $\gr_\co^{i,j}(\omega\otimes U^\vee):=\bigoplus_{\tau\in\co} \gr_\tau^{i,j}(\omega\otimes U^\vee).$ 
We also set $\uo^2:=\gr(\omega^2)$, i.e., 
\[\uo^2=(\bigoplus_{\tau\in\T} \uo_\tau\otimes\U_\tau^\vee)^{\lambda-sym}\cong\bigoplus_{\tau\in\T_0}\uo_\tau\otimes\U_\tau^\vee,\] where
$\uo_\tau\otimes\U_\tau^\vee = \bigoplus_{0\leq i,j\leq s_\tau}\gr_\tau^{i,j}(\omega\otimes U^\vee).$

\begin{remark}
By construction, for each $\tau \in \T$, the sheaf $gr_\tau^{s_\tau,0} (\omega\otimes U^\vee)$ arises as a quotient of $\omega_\tau\otimes U^\vee_\tau$. 
\end{remark}

\begin{prop}
\label{GRVANISH}
For each orbit $\co$, and integers $i,j \in\{0,\dots,,s_\co\}$, each sheaf  $\gr^{i,j}_\co(\uo^2)$ vanishes for all $i\leq j$ and is locally free of rank 
$(F_i-F_{i+1})(F_j-F_{j+1})(a_i-a_j)$ for all $j<i$.
\end{prop}

\begin{proof}
As the sheaves we consider are locally free, it suffices to prove the statement locally at a point $x_0\in\cS_1$.
Fix $\co$, write $s=s_\co$. It follows from the properties of the slope filtration that for each 
$\tau \in\co$ and $i,j=0,\ldots, s$, the sheaf $\gr^i(\omega)_\tau$ vanishes for  $\cf(\tau)< F_i$, and $\gr^i(\omega)_\tau= \gr^i(M)_\tau$ otherwise. Similarly, $\gr^j(U)_\tau$ vanishes for  $\cf(\tau)\geq F_j$, and  $\gr^j(U)_\tau= \gr^j(M)_\tau$  otherwise.
Thus, we deduce that 
$\gr_\tau^{i,j}(\uo^2):=\gr^i(\omega)_\tau \otimes_{\CO_\Ig} \gr^j(U)_\tau $ vanishes unless $F_j> \cf(\tau)\geq F_i$, in which case it is locally free of rank $(F_i-F_{i-1})(F_j-F_{j+1})=m_im_j$. In particular,  
$\gr_\co^{i,j}(\uo^2)=\bigoplus_{\tau\in\co} \gr_\tau^{i,j}(\uo^2)$ vanishes unless $j<i$,  in which case it is locally free of rank 
$(F_i-F_{i+1})(F_j-F_{j+1})(a_i-a_j)$, as $a_i-a_j=\#\{\tau\in\co| F_i\leq\cf(\tau)< F_j\}.$
\end{proof}

\begin{remark}\label{KS-sub}
Let $\co\in\mathfrak{O}_0$. Assume $\co\neq \co^*$. Then $\co\subset \T_0$, and the sheaf 
\[\gr_\co(\omega\otimes U^\vee):= \bigoplus_{\tau\in\co} \uo_\tau\otimes\U_\tau^\vee=\bigoplus_{0\leq j<i\leq s_\co} \gr_\co^{i,j}(\omega\otimes U^\vee)\] 
is as a direct summand of $\uo^2$, i.e., 
\[\gr_\co(\omega\otimes U^\vee)\cong \left(\gr_\co(\omega\otimes U^\vee)\bigoplus \gr_{\co^*}(\omega\otimes U^\vee)\right)^{\lambda-sym}\subset \uo^2.\] 

Assume $\co= \co^*$. Then $\co \not\subset \T_0$, and the sheaf
\[\gr_\co(\omega\otimes U^\vee)^{\lambda-sym}\cong \bigoplus_{\tau\in\co\cap\T_0} \uo_\tau\otimes\U_\tau^\vee\]
is as a direct summand of $\uo^2$.
In particular, the subsheaf of $\gr_\co(\omega\otimes U^\vee)^{\lambda-sym}$,
\[\gr^{\leq s_\co/2}_\co(\omega\otimes U^\vee):= \bigoplus_{0\leq j<i\leq s_\co/2} \gr_\co^{i,j}(\omega\otimes U^\vee), \]
is also a direct summand of $\uo^2.$  Indeed,  for any $i,j$, $0\leq j<i\leq s_\co/2$, we have 
\begin{align*} 
\gr_\co^{i,j}(\omega\otimes U^\vee)  &\cong  (\gr_\co^{i,j}(\omega\otimes U^\vee)\bigoplus \gr_\co^{s-j,s-i}(\omega\otimes U^\vee))^{\lambda-sym}\\
&\cong \bigoplus_{\tau\in\co\cap\T_0}\left(\gr_\tau^{i,j}(\omega\otimes U^\vee)\bigoplus \gr_\tau^{s-j,s-i}(\omega\otimes U^\vee)\right).\end{align*}

Similarly, the subsheaf of  $\gr_\co(\omega\otimes U^\vee)^{\lambda-sym}$   
\[\gr_\co^\Delta(\omega\otimes U^\vee):=\bigoplus_{0<j<s_\co/2} \gr_\co^{s-j,j}(\omega\otimes U^\vee)^{\lambda-sym}\]
is a direct summand of $\uo^2$. Indeed,  for an $j$, $0\leq j<s_\co/2$,
\[\gr_\co^{s-j,j}(\omega\otimes U^\vee)^{\lambda-sym} \cong \bigoplus_{\tau\in \co\cap\T_0}  \gr_\tau^{s-j,j}(\omega\otimes U^\vee).\]
\end{remark}

\begin{rmk}
If all the primes of $F_0$ above $p$ split in $F$, then each orbit $\co$ satisfies $\co\neq \co^*$, and \[\uo^2=\oplus_{\co\in\mathfrak{O}_0} \gr_\co(\omega\otimes U^\vee).\]
Yet, in general, 
\[\uo^2\supsetneq \bigoplus_{\co\in\mathfrak{O}_0,\, \co\neq \co^*} \gr_\co(\omega\otimes U^\vee) \oplus \bigoplus_{\co\in\mathfrak{O}_0,\, \co=\co^*} \left(\gr^{\leq s_\co/2}_\co(\omega\otimes U^\vee) \oplus \gr_\co^\Delta(\omega\otimes U^\vee) \right).\]

\end{rmk}

\subsubsection{Serre--Tate theory in the $\mu$-ordinary case}\label{ST-mu-LT}
In \cite[Sections 2.2 and 2.3]{moonen}, Moonen describes the $\mu$-ordinary local  EL moduli (i.e., the associated unpolarized deformation problem) as a cascade of Barsotti--Tate groups over $\W$. 
More precisely,  given an orbit $\co\in\mathfrak{O}$, for each pair of distinct slopes $a, b$, $ a> b$, of the $\mu$-ordinary Newton polygon  $\nu_\co$,  together with their multiplicities, Moonen defines a Barsotti--Tate group $\bG_{a,b}/\W$, and  proves that the local EL moduli corresponding to $\co$ has a natural structure of a cascade of biextensions of the groups $\bG_{a,b}/\W$, for all $a,b$.

\begin{defn} (\cite[Section 2.3.2]{moonen})
For an orbit $\co\in\mathfrak{O}$, and two distinct slopes $a=a_i>b=a_j$  of $\nu_\co$, $0\leq j<i\leq s_\co$,
the Barsotti--Tate group $\bG_{a,b}$ over $\W$  is defined as $$\bG_{a,b}=\bG_{a_i,a_j}:=\X^{\rm can}(\co, 1,\cf'_{i,j})^{m_im_j},$$
where $\X^{\rm can}(\co, 1,\cf'_{i,j})$ is the canonical lifting (in the sense of  Proposition \ref{canonicallift}) of the $\mu$-ordinary $\CO_{F,u_\co}$-module $\X(\co,1,\cf'_{i,j})$, and $m_i,m_j$ denote respectively the multiplicities of $a_i,a_j$ (with the notation of the proof of Proposition \ref{localsplit}, $m_l=F_l-F_{l+1}$, for all $l=0, \dots, s_\co$). 

The \BT group $\X(\co,1,\cf'_{i,j})$ is an isoclinic $\CO_{F,u_\co}$-module of dimension $a_i-a_j$  and height $ e $. 
(For the definition of the multiplicative type $\cf'_{i,j}:\co\rightarrow \{0,1\}$ see \cite[Section 2.3.2]{moonen}. Also, in {\it loc. cit.} the group $\bG_{a,b}=\bG_{a_i,a_j}$ is denoted by $G^{(j,i)}$ and the multiplicities $m_i$ by $ d^i$.)
\end{defn}

In the inert case for $s_\co=1$, i.e., in the case of one orbit $\co$ and two distinct slopes $a,b$, $a>b$, Moonen's result  (\cite[Theorem 2.3.3]{moonen}) states the local EL moduli is isomorphic (as a group) to the Barsotti--Tate $\CO_{F,u_\co}$-module $\bG_{a,b}$, where the natural group structure of the local EL moduli is defined by its identification with the space of extensions of $(\X^2)^\can$ by $(\X^1)^\can$ (the identity of the group corresponding to the canonical split lifting $\X^\can=(\X^1)^\can\oplus (\X^2)^\can$). In the general case, the existence of a cascade structure is defined by induction on the number of slopes, and separately for each orbit.
In particular, the cascade structure of the local EL moduli implies that, for each orbit $\co$ and pair $i,j$, $0\leq j<i\leq s_\co$, the subspace of the local EL moduli, corresponding to partially split extensions of the type 
$X(\co) \oplus\left(\oplus_{l\neq i,j}(\X(\co)^l)^\can\right)$, where $X(\co)$ is an extension of the $\CO_{F,u_\co}$-module  $(\X(\co)^j)^\can$ by $(\X(\co)^i)^\can$, is isomorphic to the Barsotti--Tate group $\bG_{a_i,a_j}$, where $a_i,a_j$ denote respectively the $i$-th and $j$-th slopes of $\nu_\co$.

In general, the local PEL moduli (which can be realized as a subspace of the local EL moduli) does not  inherit a cascade structure \cite[Section 3.3.2]{moonen}. To be more precise, let us distinguish the cases of  $\co\neq \co^*$ and $\co=\co^*$.   

If $\co\neq\co^*$, then the local PEL moduli associated with the pair $(\co,\co^*)$ is canoncally isomorphic to the local EL moduli associated with $\co$, and thus also has a natural cascade structure.

If $\co=\co^*$, then the local PEL moduli does not have a cascade structure in general, although the following weaker statements hold.
For each pair $i,j$, $0\leq j<i\leq s_\co/2$, the subspace of the local PEL moduli corresponding to partially split self-dual extensions of the type 
$X(\co)\oplus X(\co)^* \oplus\left(\oplus_{l\neq i,j,s-i,s-j}(\X(\co)^l)^\can\right)$, where $X(\co)$ is an extension of the $\CO_{F,u_\co}$-modules  $(\X(\co)^j)^\can$ by $(\X(\co)^i)^\can$, is isomorphic to the Barsotti--Tate group $\bG_{a_i,a_j}$, where $a_i,a_j$ denote respectively the $i$-th and $j$-th slopes of $\nu_\co$. 

More subtly,  for each $j$, $0\leq j< s_\co/2$, the subspace corresponding to partially split symmetric extensions of the type 
$X(\co)\oplus\left(\oplus_{l\neq j,s-j}(\X(\co)^l)^\can\right)$, where $X(\co)$ is a self-dual extension of the $\CO_{F,u_\co}$-modules  $(\X(\co)^j)^\can$ by $(\X(\co)^{s-j})^\can$, is isomorphic to a sub-$\CO_{F_0,v_\co}$-module  $\bG'_{a_{s-j},a_j}$ of  the $\CO_{F,u_\co}$-module $\bG_{a_{s-j},a_j}$, where $a_j$ denotes the $j$-th slope of $\nu_\co$ (\cite[Section 3.3.2]{moonen}). 

\begin{rmk}
As $\bG'_{a_{s-j},a_j}$ is a Barsotti--Tate subgroup of $\bG_{a_{s-j},a_j}$, it is also isoclinic of the same slope as Barsotti--Tate groups.  (See Remark \ref{BTslope}.)  Thus, as an $\CO_{F_0,v_\co}$-module, it has slope $a_{s-j}-a_j/2$. (Recall $a_{s-j}=e-a_j$, thus $a_{s-j}-a_j/2=e/2-a_j$.)
\end{rmk}

\begin{remark}
In classical Serre--Tate theory, i.e., for $\X$ ordinary and $g$-dimensional, the local EL moduli space parametrizes extensions of $(\Q_p/\Z_p)^g$
by $\mu_{p^\infty}^g$, and is isomorphic to $\hat{\mathbb G}^{g^2}_m$, while  the local PEL moduli space, which corresponds to the subspace of self-dual extensions, is isomorphic to $\hat{\mathbb G}^{g(g+1)/2}_m$.
\end{remark}

Abusing notation, in the following, we simply write $\bG_{a,b}/\W$ in place of $\bG'_{a,b}/\W$ when appropriate.

\subsubsection{The Kodaira--Spencer isomorphism}
Fix a point $x\in \Ig(\F)$, we write $\cR=\CO_{\Ig,x}^\wedge$. The canonical ($\CO_F\otimes \cR$-linear) splittings $\omega_{x}\cong \uo_{x}$ and $U_x=\U_{x}$ induce an isomorphism $\omega^2_x\cong \uo^2_x$.
We denote by $\ks_{x}$ the composition with the localization at $x$ of (the pullback over $\Ig$ of) the Kodaira--Spencer isomomorphism 
$\ks:\omega^2\to \Omega^1_{\Ig/\W}$ with the canonical splitting $\omega^2_x\cong \uo^2_x$, i.e.
$$\ks_{x}: \uo_{x}^2 
\cong \omega^2_x\to \Omega^1_{\cR/\W},$$
and still refer to it as the (split) localization at $x$ of the Kodaira--Spencer isomomorphism.

We deduce the following result from \cite[Theomre 2.3.3]{moonen}.

\begin{prop} With the notations of remark \ref{KS-sub}. Let $x$ be a point of $\Ig$, and $\co\in\mathfrak{O}_0$.
\begin{enumerate}
\item
Assume $\co\neq \co^*$. Then for  each pair of integers $i,j$, $0\leq j<i\leq s_\co$, the Kodaira--Spencer isomorphism $\ks$ induces local isomorphisms
$$\ks_{x,\co}^{ i,j}: \gr_\co^{i,j}(\omega\otimes U^\vee)_x \to \Omega^1_{\bG_{a_i,a_j}/\W}\otimes_{\CO_{\bG_{a_i,a_j}}}\CO_{\Ig,x}.$$
\item 
Assume $\co=\co^*$.Then for  each pair of integers $i,j$, $0\leq j<i\leq s_\co/2$, the Kodaira--Spencer isomorphism $\ks$ induces local isomorphisms
$$\ks_{x,\co}^{ i,j}: \gr_\co^{i,j}(\omega\otimes U^\vee)_x
\to \Omega^1_{\bG_{a_i,a_j}/\W}\otimes_{\CO_{\bG_{a_i,a_j}}}\CO_{\Ig,x}.$$

\item 
Assume $\co=\co^*$.
Then for  each integer $j$, $0\leq j< s_\co/2$, the Kodaira--Spencer isomorphism $\ks$ induces local isomorphisms
$$\ks_{x,\co}^{ s-j,j}: \gr_\co^{s-j,j}(\omega\otimes U^\vee)_x^{\lambda-sym} 
\to \Omega^1_{\bG_{a_{s-j},a_j}/\W}\otimes_{\CO_{\bG_{a_{s-j},a_j}}}\CO_{\Ig,x}.$$
\end{enumerate}
\end{prop}

As in \cite[Theorem 4.4.1]{KatzST}, the compatibility between the Gauss--Manin connection and the Frobenius map,
i.e., the equality $\nabla\circ F=(f\otimes {\rm id})\circ F$, 
implies  the result below.
In the following,   $\omega_{\bG_{a_i,a_j}/\W}$ denotes the space of invariant differentials of $\bG_{a,b}/\W$. 

\begin{prop} Maintain the above notation. Fix $x\in\Ig(\F)$.
Let $\tau\in\T_0$, and $i,j\in\Z$.  Assume $0\leq j<i\leq s_\tau$ if $\co_\tau\neq \co_\tau^*$, and assume either $0\leq j<i\leq s_\tau/2$ or $0\leq j\leq s_\tau/2$ and $i=s-j$ if $\co_\tau=\co_\tau^*$.
For any $l\in \gr^i(\omega_\X)_\tau\otimes_\W \gr^j(\omega_{\X^\vee})_\tau$, 
\[\ks_x(\alpha\otimes\alpha^\vee(l))\in\omega_{\bG_{a_i,a_j}/\W}.\]
\end{prop}
\begin{proof}

Note that the image under $\alpha\otimes \alpha^\vee$ of the space $\gr^i(\omega_\X)_\tau\otimes_\W \gr^j(\omega_{\X^\vee})_\tau$ is
a $\W$-lattice in $\gr_\tau^{i,j}(\uo^2)$.  That is, \[\gr_\tau^{i,j}(\uo^2)=(\alpha\otimes \alpha^\vee)\left(\gr^i(\omega_\X)_\tau\otimes \gr^j(\omega_{\X^\vee})_\tau)\right)\otimes_\W\CO_\Ig.\]   

Similarly, the space $\omega_{\bG_{a_i,a_j}/\W}$ of invariant differentials of $\bG_{a,b}/\W$ is a $\W$-lattice in $\Omega^1_{\bG_{a,b}/\W}\otimes_{\CO_{\bG_{a,b}}}\CO_\Ig$.  That is, 
\[\Omega^1_{\bG_{a,b}/\W}\otimes_{\CO_{\bG_{a,b}}}\CO_\Ig= \omega_{\bG_{a_i,a_j}/\W}\otimes_\W\CO_\Ig.\] 

Both $\W$-lattices are characterized by the property that they admit a basis over $\W$ on which the $e$-th iterate of Frobenius $F^e$ acts as $p^{a_i-a_j}$. (For $(\alpha\otimes\alpha^\vee)\gr^i(\omega_\X)\otimes_\W \gr^j(\omega^\vee_\X)$, this basis arises from the basis of $\omega_\X$ defined in \cite[Section1.2.3]{moonen} and introduced in the proof of Propostion \ref{GMST}.)
Thus, the statement follows from the equality $\nabla\circ F^e=(f\otimes {\rm id})\circ F^e$.
\end{proof}

 \begin{rmk}\label{MULT}
 Fix an orbit $\co$, let $\cf$ be the associated  multiplicative type. 
Note that $0$ is a slope of $\X(\co,n,\cf)$  if and only if $\cf(\tau)\neq n$  for all $\tau\in\co$.  (e.g.,  $\X(\co,n,\cf)$  is \'etale if $\cf(\tau)=0$, for all $\tau\in\co$.)
Similarly, $e$ is a slope of $\X(\co,n,\cf)$ if and only if  $\cf(\tau)\neq 0$, for all $\tau\in\co$. (e.g.,  $\X(\co,n,\cf)$  is multiplicative if $\cf(\tau)=n$, for all $\tau\in\co$.)

Assume both $0$ and $e$ are slopes of $\X(\co,n,\cf)$ (i.e., for all $\tau\in\co$, $\cf(\tau)\neq 0,n$).  Then $a_1=0$ and $a_s=e$, and the Barsotti--Tate group $\bG_{0,e}$ occurs in the cascade of the local EL moduli, $\bG_{0,e}= \X^{\rm can}(\co, 1,\cf'_{0,e})^{d^1d^s}$. By definition, $\X^{\rm can}(\co, 1,\cf'_{0,e})$ is isomorphic to a sum of $e$-copies of the formal multiplicative group $\hat{\bG}_m$. In fact,  the given condition is both sufficient and necessary for the formal multiplicative groups to occur in the cascade. 
 
For  $u_1, \dots, u_e$ a choice of parameters of $\X^{\rm can}(\co, 1,\cf'_{0,e})/\W$ (i.e.,  $\CO_{\X^{\rm can}(\co, 1,\cf'_{0,e})}=\W[\![u_1,\dots u_e]\!]$), the space of invariant differentials of $\X^{\rm can}(\co, 1,\cf'_{0,e})$ is generated by 
\[\eta_i:=dlog(q_i)=\frac{1}{q_i}dq_i\in\omega_{\X^{\rm can}(\co, 1,\cf'_{0,e})/\W}\subset \Omega^1_{{\X^{\rm can}(\co, 1,\cf'_{0,e})}}=\langle du_1, \dots ,du_e\rangle_{\W[\![u_1,\dots, u_e]\!]},\] for $q_i =1+ u_i$, $i=1, \dots, e$.

 \end{rmk}

 \begin{rmk}\label{LT}
Fix an orbit $\co$, let $\cf$ be the associated  multiplicative type. Note that  there exists an integer 
$a\in\{0,\dots, e\}$ such that both $a,a+1$ are slopes of $\X(\co,n,\cf)$ if and only if there exists $\tau_0\in\co$ such that $\cf(\tau_0)\neq \cf(\tau)$ for all $\tau\neq \tau_0$.

Assume both $a$ and $a+1$ are slopes of $\X(\co,n,\cf)$, for some integer $a$, $0\leq a\leq e$. Them, the Barsotti--Tate group $\bG_{a,a+1}$ occurs in the cascade of the local moduli. By definition, $\X^{\rm can}(\co, 1,\cf'_{a,a+1})$ is a formal Lubin--Tate $\CO_{F,u_\co}$-modules of slope $1/e$. In fact, the given condition is both sufficient and necessary for a formal Lubin--Tate  $\CO_{F,u_\co}$-module to occur in the cascade. 

For $u$ a choice of a parameter of $\X^{\rm can}(\co, 1,\cf'_{a,a+1})/\W$ (i.e., $\CO_{\X^{\rm can}(\co, 1,\cf'_{a,a+1})}=\W[\![u]\!]$),  the space of invariant differentials of $\X^{\rm can}(\co, 1,\cf'_{a,a+1})$ is generated by \[\eta:=dlog_{G}(u):=G_x(0,u)^{-1}du \in\Omega^1_{\X^{\rm can}(\co, 1,\cf'_{a,a+1})/\W}=\langle du\rangle_{\W[\![u]\!]}\] where $G(x,y)$ denotes the formal group law of $\X^{\rm can}(\co, 1,\cf'_{a,a+1})$ with respect to the choosen parameter $u$, and 
$G_x$ denotes the partial derivative of $G$ with respect to the variable $x$ (\cite[p.\ 4]{JaredWeinsteinLT}).
 \end{rmk}

\section{Differential operators in the $\mu$-ordinary setting}\label{do-mo}
This section introduces differential operators that enable construction of new $p$-adic automorphic forms and families.  Unlike in the ordinary setting in \cite{kaCM, EDiffOps, EFMV}, we now need to keep track of slope filtrations and rely on new results about the canonical complement to $\omega$ introduced in Section \ref{ur-section}.

\subsection{Definition of $p$-adic differential operators}\label{diffop}
As our construction of $p$-adic differential operators begins similarly in the ordinary setting \cite{kaCM, EDiffOps, E09, EFMV}, we focus here primarily on the details unique to the $\mu$-ordinary case, namely the roles of the filtration and semi-simplification, which pose additional challenges.  This is one of the most challenging parts of this paper.

Following the conventions of \cite[Section 3.3]{EFMV}, for all positive integers $d$ and $e$ and any irreducible representation $\rho:=\rho_\kappa$ of highest weight $\kappa$, we define morphisms of sheaves
\begin{align*}
\nabla_{\otimes d}^e:&\hdrone\left(A/S\right)^{\otimes d}\rightarrow\hdrone\left(A/S\right)^{\otimes d}\otimes\left(\bigoplus_{\tau\in\T_0}\left(\hdr\left(A/S\right)_\tau\otimes\hdr\left(A/S\right)_{\tau^*}\right)\right)^{\otimes e}\\
\nabla_\rho^e:=\nabla^e_\kappa:& \schur_\kappa\left(\hdrone\left(A/S\right)\right)\rightarrow\schur_\kappa\left(\hdrone\left(A/S\right)\right)\otimes
\left(\bigoplus_{\tau\in\T_0}\left(\hdr\left(A/S\right)_\tau\otimes\hdr\left(A/S\right)_{\tau^*}\right)\right)^{\otimes e},
\end{align*}
where
\begin{align}\label{subscript-equation}
\nabla_{\otimes d}^e:=\nabla_{\otimes d+2(e-1)}\circ \cdots \circ\nabla_{\otimes d},
\end{align}
$\nabla_{\otimes d}$ denotes the Gauss--Manin connection extended to $\hdrone\left(A/S\right)^{\otimes d}$ via the product rule (Leibniz's rule), and $\nabla_{\kappa}^e$ is the morphism induced by $\nabla_{\otimes d_\kappa}^e$.  We also define
\begin{align*}
\nabla_\kappa:=\nabla_\rho&:=\nabla_\rho^1
\end{align*}
Observe that $\nabla_\rho=\nabla_\kappa$ decomposes as a direct sum of morphisms
\begin{align*}
\nabla_\rho(\tau):\left(\hdrone\left(A/S\right)\right)^\rho\rightarrow\left(\hdrone\left(A/S\right)\right)^\rho\otimes
\left(\hdr\left(A/S\right)_\tau\otimes\hdr\left(A/S\right)_{\tau^*}\right).
\end{align*}

For each positive integer $e$, we define $\nabla_\kappa^e\left(\tau\right):=\nabla_\rho^e\left(\tau\right)$ to be the composition of 
$\nabla_\kappa\left(\tau\right):=\nabla_\rho\left(\tau\right)$ 
with itself $e$ times (with the subscript increasing as in Equation \eqref{subscript-equation}).

We also denote by $\Auniv$ the pullback of the universal abelian scheme $\Auniv/\Smu$ over $\Igmu$.  For each irreducible representation $\rho$, the splitting
\begin{align*}
\hdrone\left(\Auniv/\Igmu\right) = \omega_{\Auniv/\Igmu}\oplus U
\end{align*}
induces a projection
\begin{align*}
\varpi\Aoverig: \left(\hdrone\left(\Auniv/\Igmu\right)\right)^\rho\twoheadrightarrow\left(\omega_{\Auniv/\Igmu}\right)^\rho
\end{align*}
(projection modulo $U$).

We define
\begin{align*}
D_\rho^e\Aoverig:& \left(\omega\Aoverig\right)^\rho\rightarrow\left(\omega\Aoverig\right)^\rho\otimes
\left(\oplus_{\tau\in\T_0}\left(\omega_\tau\Aoverig\otimes\omega_{\tau^*}\Aoverig\right)\right)^e
\end{align*}
by
\begin{align*}
D_\rho\Aoverig & = \varpi\Aoverig\circ\nabla_\rho^e.
\end{align*}
We define
\begin{align*}
D_\rho\Aoverig :=D_\rho^1\Aoverig.
\end{align*}

When it is clear from context that we are working with $\Auniv/\Igmu$, we simply write $D_\rho^e$, $D_\rho$, etc.  For the other operators  introduced below, we follow similar conventions with regard to inclusion of $\Auniv/\Igmu$ in the notation.
Since $\nabla(\unitrootold)\subseteq \unitrootold\otimes\Omega_{\Auniv/\Igmu}$, we have that 
\begin{align*}
D_\rho^e = D_{\rho\otimes\left(\st\otimes\st\right)^{e-1}}\circ\cdots \circ D_{\rho\otimes\left(\st\otimes\st\right)}\circ D_\rho.
\end{align*}

For any irreducible representation $\CZ$ that is sum-symmetric of depth $e$, consider the projection
\begin{align*}
\pi_{\CZ}: \left(\bigoplus_{\tau\in\T_0}\left(\omega_\tau\Aoverig\otimes\omega_{\tau^*}\Aoverig\right)\right)^e\twoheadrightarrow \left(\omega\Aoverig\right)^{\CZ}.
\end{align*}
Also, for $\kappa$ a positive dominant weight, $\rho_\kappa$ an irreducible representation of highest weight $\kappa$, $\kappa'$ sum-symmetric of depth $e$, and $\CZ:=\rho_{\kappa'}$, consider the projection
\begin{align*}
\pi_{\kappa, \kappa'}: \CE_{\rho\otimes\CZ}\rightarrow\CE_{\rho_{\kappa\cdot\kappa'}}
\end{align*}
induced by the canonical projection (described in more detail in \cite[Lemma 2.4.6]{EFMV})
\begin{align*}
\rho_\kappa\otimes\rho_{\kappa'}\twoheadrightarrow\rho_{\kappa\cdot\kappa'}.
\end{align*}
We define
\begin{align*}
D_\rho^{\CZ}\Aoverig &:=\left(\id\otimes\pi_{\CZ}\right)\circ D_\rho^e\\
D_\kappa^{\kappa'}\Aoverig &:=\pi_{\kappa, \kappa'}\circ D_\rho^{\CZ}
\end{align*}

\begin{rmk}
The $p$-adic operators $D_\rho^\CZ$ are the $p$-adic incarnation of the Maass--Shimura $\ci$ differential operators $D_{\rho, \ci}^\CZ$ that arise over $\IC$ and are described in detail in, for example, \cite[Section 23]{sh}, \cite{shclassical}, and \cite[Section 12]{Shimura}.  The construction of these $\ci$ differential operators is similar, except that the complement to $\omega$ is replaced by the Hodge de Rham splitting.  The Hodge theoretic construction was first completed by M.\ Harris for Siegel modular forms in \cite{hasv} (which we recommend to readers trying to get acquainted with the ideas of this construction) and for more general Shimura varieties in \cite{ha86}.  We also define analogous $\ci$ differential operators $D_{\kappa, \ci}^{\kappa'}$ similarly.
\end{rmk}

\subsection{$p$-adic differential operators on $p$-adic forms in the $\mu$-ordinary setting}

We consider the sequence of sheaves over $\cS$
$$0\rightarrow\omega \subseteq {\hdr^1} \rightarrow \omega^\vee\rightarrow 0,$$
together with its canonical splitting ${\hdr^1}\supseteq U\cong \omega^\vee$, constructed in Proposition \ref{unitroot}.
We denote by
$\varpi:{\hdr^1}\to{\omega}$ the projection modulo $U$, and write
$$D:=(\varpi\otimes{\rm id})\circ \nabla:\omega\subseteq \hdr^1\to\hdr^1\otimes\Omega^1_{\CS/\W}\to \omega \otimes\Omega^1_{\CS/\W},$$ 
and
$$D_\kappa=\schur_\kappa(D): \omega^\kappa\to\omega^\kappa \otimes \Omega^1_{\CS/\W},$$
for any dominant weight $\kappa$ of $\Levi$.

Abusing notation, we still denote by $D_\kappa$ the composition of $D_\kappa$ with $\ks^{-1}$, the inverse of the Kodaira--Spencer isomorphism.

Finally, for each sum-symmetric weight $\lambda$ of $\Levi$, of depth $e$, we write 
$$D_\kappa^\lambda:=\pi_{\kappa,\lambda}\circ ({\rm id}\otimes\pi_\lambda)\circ  D_\kappa^e:\omega^\kappa\to\omega^\kappa \otimes (\omega^2)^{\otimes e}\to \omega^\kappa\otimes \omega^\lambda\to\omega^{\kappa+\lambda}
$$ 
for any dominant weight $\kappa$ of $\Levi$.

\begin{prop}\label{Dfilter}
The operator
$$D:\omega\subseteq \hdr^1\to\hdr^1\otimes\Omega^1_{\CA/\CS}\to \omega \otimes\Omega^1_{\CS/\W}$$ 
preserves the canonical decomposition $\omega=\oplus_\co \omega_\co$, and the filtration $\omega_\bullet$ induced by the slope filtration of $\hdr^1$. That is, for each orbit $\co\in\FO$, and each slope $a$ of the subcrystal $M_\co$,
$$D\left((\omega_\co)_a\right)\subset (\omega_\co)_a\otimes \Omega^1_{\CS/\W}.$$
\end{prop}
\begin{proof}
It follows from \cite[Proposition 2.1.9]{moonen} that the Gauss--Manin connection $\nabla$ of $\hdr^1$ preserves the subcrystals $(M_\co)_a$, for each orbit $\co\in\FO$ and each slope $a$ of $M_\co$. Thus, the statement follows from Part \eqref{ur3} of Proposition \ref{unitroot}.
\end{proof}

Note that, by construction, the operator $D$ is $\W$-linear. Thus, for each $\tau\in \T$, 
$D\left(\omega_\tau\right)\subset \omega_\tau \otimes \Omega^1_{\CS/\W}.$
Proposition \ref{Dfilter} implies the operator $D$ induces a graded operator $\underline{D}$ on $\uo=\gr(\omega)$. That is,
$$\underline{D}=\gr(D): \uo \to \uo\otimes \Omega^1_{\CS/\W}.$$

Similarly to the construction in Section \ref{diffop}, starting from the differential operator $\ud:\uo\to \uo\otimes\Omega^1_{\CS/\W}$, 
we may construct new $p$-adic differential operators on the 
sheaves $\uo^{\kappa'}$ over the $\mu$-ordinary Igusa tower.

\begin{defi}\label{ud}
Let $\kappa'$ be a dominant weight of $\levi$.  We define the differential operator
$$\ud_{\kappa'}:=\schur_{\kappa'}(\ud): \uo^{\kappa'}\to\uo^{\kappa'} \otimes \Omega^1_{\CS/\W}.$$

Abusing notation, we also denote by $\ud_{\kappa'}$ the composition of $\ud_{\kappa'}$ with the inverse $\ks^{-1}$ of the Kodaira--Spencer isomorphism.  That is, 
\[\ud_{\kappa'}: \uo^{\kappa'}\to\uo^{\kappa'} \otimes \omega^2. \]

For each 
sum-symmetric weight $\lambda'$ of $\Levi$ of depth $e$, we  define
$$\ud_{\kappa'}^{\lambda'}:=\pi_{\kappa',\lambda'}\circ ({\rm id}\otimes\pi_{\lambda'})\circ  \ud_{\kappa'}^e:\uo^{\kappa'}\to\uo^{\kappa'} \otimes (\omega^2)^{\otimes e}\to \uo^{\kappa'}\otimes \omega^{\lambda'}\to \uo^{\kappa'}\otimes \uo^{\lambda'}\to\uo^{\kappa'+\lambda'}.
$$ 
\end{defi}

\subsubsection{Differential operators locally}\label{DisDU}
Fix $x_0\in\CS(\F)$. Via the canonical splitting $\omega_{x_0}\cong \uo_{x_0}$ constructed in Proposition \ref{localsplit}, we obtain a decomposition of $D_{x_0}$ into blocks. Proposition \ref{Dfilter} implies that $D_{x_0}$ is block upper triangular, with $\ud_{x_0}$ on the block diagonal.  That is, we have a canonical factorization
$$D_{x_0}=\ud_{x_0}\circ\unipo_{x_0},$$ 
where $\unipo_{x_0}$ is unipotent block upper triangular.  

In the next section (Proposition \ref{DisD}), we establish the equality $D_{x_0}=\ud_{x_0}$, which implies the following decompositions of differential operators.
\begin{prop}\label{DKloc} Mantain the above notation. 
\begin{enumerate}
\item For any dominant  weight $\kappa$ of $\Levi$, 
\[D_{\kappa,x_0}=\bigoplus_{\kappa'\in\M_\kappa} \ud_{\kappa',x_0}. \]
\item For any dominant weight $\kappa$ and  sum-symmetric weight $\lambda$ of $\Levi$,
the morphism $D_{\kappa,x_0}^\lambda$ decomposes as a direct sum of morphisms
\[ {\mathcal D}^{\lambda'}_{\kappa'}:\uo_{x_0}^{\kappa'}\to \uo_{x_0}^{\kappa'+\lambda'}\]
for all $\kappa'\in\M_\kappa$, and $\lambda'\in\M_{\lambda}$ such that
$\kappa'+\lambda'\in\M_{\kappa+\lambda}$.

Furthermore, for all  $\kappa'\in\M_\kappa$,
\[{\mathcal D}^{\lambda}_{\kappa'}=\ud^{\lambda}_{\kappa',x_0}.\] 
 
In  particular, if $\lambda$ is a positive scalar weight, we have
\[D^{\lambda}_{\kappa,x_0}=\bigoplus_{\kappa'\in\M_\kappa} \ud^{\lambda}_{\kappa',x_0}.\]
\end{enumerate}
\end{prop}

\begin{proof}
The first equality follows from the compatibility among the projections \[\omega^{\otimes d}\to\schur_\kappa(\omega)\text{ and }\uo^{\otimes d}\to \uo^{\otimes \underline{d}}\to \schur_{\kappa'}(\uo).\] 
Note that for each $\kappa'\in\M_\kappa$, $|\kappa'|=\underline{d}$ is a partition of $d=|\kappa|$.
Similarly, the second equality follows from the compatibility among the projections \[(\omega^2)^{\otimes e}\to \omega^\lambda\to\uo^\lambda \text{ and }
(\uo^2)^{\otimes e}\to\uo^{\lambda}.\] 
Finally, to deduce the last equality if suffices to recall that $\M_\lambda=\{\lambda\}$ when $\lambda$ is scalar.
\end{proof}

\begin{cor}\label{compare}\label{DKpi}
Maintain the above notation and assumptions.
Let $\kappa_1,\kappa_2$ be two 
dominant weights of $\Levi$. Assume $\kappa_2-\kappa_1$ is  
sum-symmetric. Then  for any automorphic form $f$ of weight $\kappa_1$,    we have
\[ \pi^{\kappa_2}D^{\kappa_2-\kappa_1}_{\kappa_1} (f)= \ud^{\kappa_2-\kappa_1}_{\kappa_1}(\pi^{\kappa_1} f).\]
In particular, if $\kappa_1$ is a scalar weight, then
\[ \pi^{\kappa_2}D^{\kappa_2-\kappa_1}_{\kappa_1} (f)= \ud^{\kappa_2-\kappa_1}_{\kappa_1}(f).\]
\end{cor}

\subsubsection{The action of the differential operators on $u$-expansions}\label{action-section}
In this section, we describe the action of the differential operators on $u$-expansions, in certain cases.  This description is crucial for our approach to constructing families of $p$-adic automorphic forms.

We fix a point $x\in\Ig(\W)$, and write $\cR$ for the complete local ring of $\Ig$ at $x$.  In the next section, we explicitly compute the action of the differential operators on Serre--Tate expansions. By abuse of notation, we will still denote by $D$ (respectively, $\ud$, $\ud_\kappa$, $\ud_\kappa^e$) the localization at $x$ of the differential operators $D$ (respectively, $\ud$, $\ud_\kappa$, $\ud_\kappa^e$), i.e., their base change to $\cR$.

For convenience, we write $\CL:=\uo_\X$, $\CL^2:=\uo^2_\X$, and $\CL^\kappa:=\schur_\kappa(\CL)$, for all $\kappa$ dominant weights of $\levi$ (all regarded as  $\W$-representations of $\levi$).

Abusing notation, we still denote by $\alpha$ the universal Igusa structure over $\cR$ composed with the canonical splitting of $\omega$, i.e., the $\cR$-linear isomorphism \[\alpha=\alpha_x:\CL\otimes _\W \cR\to\uo.\] 
For each dominant weight $\kappa$ of $\levi$, we define \[\alpha^\kappa:\CL^\kappa\otimes_\W\cR\to \uo^\kappa\] to be the $\cR$-linear isomorphism induced by $\alpha$.

We denote by $d:\cR\to\Omega^1_{\cR/\W}$ the universal derivation on $\cR$.

\begin{prop}\label{DisD}
Maintain the above notation. 
\begin{enumerate}
\item After  identifying $\omega\cong \uo$ via the canonical splitting, we have

\[(\alpha\otimes_{\CO_\Ig} {\rm id}_{\Omega^1_{\Ig/\W}} )^{-1}\circ D \circ \alpha= ({\rm id}_{\CL}\otimes_\W d).\]

In particular, we deduce $D=\ud$.

\item
For any dominant weight $\kappa$ of $\levi$, we have 

\[(\alpha^\kappa \otimes_{\CO_\Ig} {\rm id}_{\Omega^1_{\Ig/\W}} )^{-1}\circ \ud_\kappa \circ \alpha^\kappa= ({\rm id}_{\CL^\kappa }\otimes_\W d).\]\label{DisD2}
\end{enumerate}
\end{prop}
\begin{proof}
The statement follows immediately from Proposition \ref{horizontalbasis}. \end{proof}

\begin{definition}\label{XI}
We define \[\Xi:=\ks^{-1}\circ d:\cR\to\Omega^1_{\cR/\W}\cong  \CL^2\otimes_\W \cR.\]
For any integer $e\geq 1$, we write $\Xi^e:= ({\rm id}_{(\CL^2)^{\otimes e-1}}\otimes\Xi)\circ \cdots \circ \Xi:\cR\to (\CL^2)^{\otimes e}\otimes_\W \cR.$ \end{definition}

We also write
$\alpha^{\kappa,e}:=\alpha^\kappa\otimes_\cR (\alpha^2)^{\otimes e} :\left(\CL^\kappa\otimes_\W(\CL^2)^{\otimes e}\right)\otimes_\W\cR \to \uo^\kappa\otimes_\cR (\uo^2)^{\otimes e}.$

With the new notation, Part \eqref{DisD2} of Proposition \ref{DisD} implies the following description of the operators $\ud_\kappa^e$.

\begin{prop}\label{disxi}
Maintain the above notation. 
For any dominant weight $\kappa$ of $\levi$ and any  integer $e\geq 1$, we have 
\[(\alpha^{\kappa,e} \otimes_{\CO_\Ig} {\rm id}_{\Omega^1_{\Ig/\W}} )^{-1}\circ \ud^e_\kappa \circ \alpha^\kappa= ({\rm id}_{\CL^\kappa }\otimes_\W \Xi^e).\]
\end{prop}

\subsection{p-adic differential operators on $p$-adic forms OMOIT}
In this section, we introduce the $p$-adic operators that act on the space of $p$-adic automorphic forms OMOIT.  When the ordinary locus is nonempty, this operator agrees with the $p$-adic operator conventionally denoted $\Theta$ (see, e.g., \cite{kaCM, EDiffOps, E09, EFMV, dSG}).

Adapting the conventions of \cite[Section 5.1]{EFMV}  (e.g., replacing $\Levi$ with $\levi$ and the sheaves $\omega^\kappa$ with $\uo^\kappa$), we deduce an analogue of \cite[Theorem 5.1.3]{EFMV} in our context, stating the existence, for each
sum-symmetric weight $\lambda$ of $\Levi$, of a (unique) operator $\Theta^\lambda$ on $\Vmu$ which interpolates the operators $\ud_\kappa^\lambda$ from Definition \ref{ud}.

\newcommand{\tlcan}{{\tilde{\ell}}} 
\newcommand{\lcan}{{{\ell}}} 

For simplicity, abusing notation in the following, we still write $\ud_\kappa^\lambda$ for the map on global sections
\[\ud_\kappa^\lambda(\cS):H^0(\cS,\uo^\kappa)\to H^0(\cS, \uo^{\kappa+\lambda}),\]
for any $\kappa$ dominant weight of $\levi$ and $\lambda $ sum-symmetric weight of $\Levi$.

Fix $x\in\Ig(\W)$. As in Section \ref{action-section},
 we denote by $\cR$ the complete local ring of $\Ig$ at $x$, and by $\loc_x:V\to \cR$ the localization map at $x$.
 
 \begin{defn}\label{thetadefn}
For each sum-symmetric weight $\lambda$ of $\levi$, of depth $e$, we define 
\[\theta^\lambda:=(\tlcan^\lambda\otimes {\rm id}_\cR)\circ \Xi^e:\cR\to (\CL^2)^{\otimes e}\otimes_\W \cR\to\cR\]
with  $\Xi^e$ as in Definition \ref{XI}, and 
$\tlcan^\lambda:=\lcan^\lambda\circ \pi_\lambda : (\CL^2)^{\otimes e}\to\CL^\lambda \to \W$ defined similarly to \cite[Definition 2.4.2]{EFMV}, i.e., as the composition of $\lcan^\lambda$ with the projections $\pi_\lambda$ defined by the generalized Young symmetrizer $c_\lambda$. (Recall that the condition $\lambda $ sum-symmetric is to ensure that the map $\tlcan^\lambda $ is non-zero.)
\end{defn}

\begin{rmk}\label{theta}
It follows from the definitions, together with Propositions \ref{DKloc} and \ref{disxi}, that for any sum-symmetric weight $\lambda$ of $\levi$, 
\[\theta^\lambda\circ \Psi_{\kappa,x}=\Psi_{\kappa+\lambda,x}\circ {\mathcal D}_{\kappa}^{\lambda},\]
for all dominant weights $\kappa$ of $\levi$.

In particular, if $\lambda$ is a sum-symmetric weight of $\Levi$, then
\[\theta^\lambda \circ \loc_x\circ \Psi_{\kappa}= \loc_x\circ \Psi_{\lambda + \kappa}\circ \ud_\kappa^\lambda, \]
for all dominant weights $\kappa$ of $\levi$. 
\end{rmk}

\begin{thm}\label{THETA}
For each sum-symmetric weight $\lambda$ of $\Levi$, there exists a unique operator
  \begin{align*}
\Theta^\lambda: V^N\rightarrow V^N
\end{align*}
such that 
$\Theta^\lambda\circ \Psi_\kappa = \Psi_{\lambda+\kappa}\circ \ud_{\kappa}^{\lambda},$ for all dominant weights $\kappa$. 

The $p$-adic differential operator $\Theta^\lambda$ satisfies the properties
\begin{enumerate}
\item $\Theta^\lambda(\Vmu[\kappa])\subseteq \Vmu[\lambda+\kappa]$.
\item $\loc_x\circ \Theta^\lambda=\theta^\lambda\circ \loc_x$.

\end{enumerate}

\end{thm}

\begin{proof}
The argument of  \cite[Theorem 5.1.3]{EFMV} still applies here. Indeed, the injectivity of  $\Psi=\oplus_\kappa \Psi_\kappa$ allows us to define 
\[\Theta^\lambda_{\vert {\rm Im}(\Psi)}:=\Psi\circ (\oplus_\kappa \ud_\kappa^\lambda )\circ \Psi^{-1}.\] As  ${\rm Im}(\Psi)$ is dense in $V^{N_\mu(\Z_p)}$, in order to extend $\Theta^\lambda$ to $V^{N_\mu(\Z_p)}$  it suffices to check that the image under $\Theta^\lambda$ of a converging sequence in ${\rm Im}(\Psi)$ is still convergent.  This can be checked locally, by passing to $\underline{u}$-expansions, in which case the statement follows from Remark \ref{theta} . \end{proof}

\begin{cor}\label{Dscalar}
For each sum-symmetric weight $\lambda$ of $\Levi$, 
 the operator
$\Theta^\lambda: V^N\rightarrow V^N$
satisfies the equality  
\[\Theta^\lambda\circ \Phi_\kappa = \Phi_{\lambda+\kappa}\circ D_{\kappa}^{\lambda}\] for each dominant weight $\kappa$ of $\Levi$. 
\end{cor}
\begin{proof}
For each dominant weight $\kappa$ of $\Levi$, the statement follows  from Corollary \ref{DKpi} and the equalities $\Theta^\lambda\circ \Psi_{\kappa'} = \Psi_{\lambda+\kappa'}\circ \ud_{\kappa'}^{\lambda}$ for all $\kappa'\in\M_\kappa$.
\end{proof}

\subsubsection{Congruences among $p$-adic differential operators on $p$-adic automorphic forms OMOIT, via $u$-expansions}
By similarity with the theory in \cite{EFMV}, one may expect congruences among operators $\Theta^\lambda$ of congruent weights, at least under some mild/harmless assumption on the weights.
In this section, we prove that this is indeed the case under some strong restrictions on the weights (see Definition \ref{simple}). Yet, we have no reason to believe them necessary, and we have hope to improve on them in the future.

In a few cases (see Remark \ref{goodies}), 
e.g., when $p$ splits completely in the reflex field $E$, 
our assumptions reduce to the milder ones introduced in \cite{EFMV}. 
Note that in \cite{EFMV} $p$ splits completely in the field $F$, which implies, but is not equivalent to, $p$ splits completely in $E$.

\begin{definition}\label{simple}
Let $\lambda$ be a dominant weight of $\levi$, and write $\lambda=(\lambda(\co))_{\co\in\mathfrak O}$, with $\lambda(\co)=\left(\lambda(\co)_{s_\co},\dots,\lambda(\co)_1\right)$.

We call $\lambda$ {\em simple} if it is symmetric and if, for each orbit $\co$, it satisfies the following conditions:
\begin{enumerate}
\item If there exists $\tau\in\co$ satisfying $\cf_\co(\tau)\in \{0,n\}$, then $\lambda(\co)=(0,\dots, 0).$ 
\item If $\cf_\co(\tau)\neq {0,n}$ for all $\tau\in\co$, then $\lambda(\co)_i=(0,\dots, 0)$, for all $i=1,\dots, s_\co-1$. \label{simple2}
\end{enumerate}
\end{definition}

\begin{rmk} \label{manysimple}
There exist (infinitely many) non-zero simple weights if and only if there exists an orbit $\co\in\mathfrak O$ such that 
$\cf(\tau)\neq 0,n$ for all $\tau\in\co$. 
\end{rmk}

\begin{rmk}
If $\lambda$ is simple, of depth $e$, then it is a sum-symmetric (dominant) weight of $\Levi$. Moreover,  the irreducible $\W$-representation $\varrho_\lambda$ arises as a quotient of the direct summand $\gr^{s, 0}_\co(\CL^2)^{\otimes e}$ of $(\CL^2)^{\otimes e}$. 
\end{rmk}

\begin{rmk}\label{goodies}
If $p$ splits completely in the reflex field $E$, then the $\mu$-ordinary polygon is ordinary and all symmetric weights are simple.  

More generally, all symmetric weights are simple if, for each orbit $\tau\in\T$, the $\mu$-ordinary Newton polygon $\nu_{\co_\tau}(n,\cf)$  is either ordinary (i.e., its only slopes are $0$ and $e$) or isoclinic (i.e., it has only one slope).
\end{rmk}

It follows from the definitions, together with Remark \ref{MULT}, that for $\lambda$ simple, the operator $\theta^\lambda:\cR\to\cR$ can be computed as in \cite[Lemma 5.2.2]{EFMV}, in terms of operators of  $(1+u)\partial_u:\W[\![\underline{u}]\!]\to\W[\![\underline{u}]\!]$, where $u$ ranges among the Serre--Tate coordinates corresponding to the Barsotti--Tate groups $\X^\can(\co,1,\cf'_{0,e})$ in the cascades, for $\co$ as in Part \eqref{simple2} of Definition  \ref{simple}.

We deduce the following analogue of \cite[Proposition 5.2.4]{EFMV}. The argument in \cite{EFMV}  applies immediately to our setting, under the further assumptions that the two symmetric weights $\lambda,\lambda' $  are simple. (In {\it loc.\ cit.}\ the weights are denoted by $\kappa,\kappa'$.)

\begin{prop}\label{thetacong}

Let $\lambda,\lambda'$ be two simple weights, and let $m\geq 1$ be an integer.
Assume  \[\lambda\equiv\lambda'\mod p^m(p-1)\] in $\Z^g$. If, additionally, \begin{itemize}
\item $\min(\lambda(\tau)_i-\lambda(\tau)_{i+1},\lambda'(\tau)_i-\lambda'(\tau)_{i-1})>m$ for all $\tau\in\T$ and $1\leq i < a^+_{\tau}$ for which 
$\lambda(\tau)_i-\lambda(\tau)_{i+1} \neq \lambda'(\tau)_i-\lambda'(\tau)_{i-1}$, and
\item  $\min(\lambda(\tau)_{a^+_\tau},\lambda'(\tau)_{a^+_\tau})>m$ for all $\tau\in\T$  for which 
$\lambda(\tau)_{a^+_\tau}\neq \lambda'(\tau)_{a^+_\tau}$,
\end{itemize} 
 then $\theta^\lambda \equiv \theta^{\lambda'}\mod p^{m+1}$.
\end{prop}

Finally, from the above proposition and Theorem \ref{THETA} combined, we deduce the following analogue of \cite[Theorem 5.2.6]{EFMV}.
  
\begin{thm}\label{THETACONG}
Let $\lambda,\lambda'$ be two simple weights, and let $m\geq 1$ be an integer.
Assume  \[\lambda\equiv\lambda'\mod p^m(p-1)\] in $\Z^g$. If, additionally, both 
\begin{enumerate}
\item{$\min(\lambda(\tau)_i-\lambda(\tau)_{i+1},\lambda'(\tau)_i-\lambda'(\tau)_{i-1})>m$ for all $\tau\in\T$ and $1\leq i < a^+_{\tau}$ for which 
$\lambda(\tau)_i-\lambda(\tau)_{i+1} \neq \lambda'(\tau)_i-\lambda'(\tau)_{i-1}$, and}\label{itemi}
\item{$\min(\lambda(\tau)_{a^+_\tau},\lambda'(\tau)_{a^+_\tau})>m$ for all $\tau\in\T$  for which 
$\lambda(\tau)_{a^+_\tau}\neq \lambda'(\tau)_{a^+_\tau},$}\label{itemii}
\end{enumerate} 
 then $\Theta^\lambda \equiv \Theta^{\lambda'}\mod p^{m+1}$.
\end{thm}

\begin{defn}
A character $T(\Z_p)\to\Z_p^*$ is called a {\em (simple) $p$-adic character} if it can be expressed as the $p$-adic limit of a sequence of characters corresponding to (simple) classical weights.
\end{defn}

Proposition \ref{thetacong} implies the existence of differential operators $\theta^\chi$ on $\cR$ for all simple $p$-adic characters $\chi$, arising by interpolation of the operators $\theta^\lambda$, for simple weights $\lambda$ of $\levi$.  (Take a sequence of $p$-adically converging simple weights $\lambda_i$ with $|\lambda_i|_\infty\rightarrow \infty$ so that \eqref{itemi} and \eqref{itemii} from Theorem \ref{THETACONG} are satisfied.)

Similarly, Theorem \ref{THETACONG} implies the following result. 

\begin{cor}

For each simple $p$-adic character $\chi$, there exists a $p$-adic differential operator
  \begin{align*}
\Theta^\chi: V^N\rightarrow V^N
\end{align*}
interpolating the $p$-adic differential operators $\Theta^\lambda$.  That is, if $\lambda_i\rightarrow \chi$ $p$-adically and $|\lambda_i|_\infty\rightarrow\infty$ as $i\rightarrow\infty$, then $\Theta^\chi(f) =\lim_i\Theta^{\lambda_i}(f)$.
The $p$-adic differential operator $\Theta^\chi$ satisfies the following properties:
\begin{enumerate}
\item For all $p$-adic characters $\chi'$: $\Theta^\chi (V^N[\chi'])\subseteq V^N[\chi\cdot\chi']$.
\item For all $x\in\Ig(\W)$: $\loc_x\circ \Theta^\chi=\theta^\chi\circ \loc_x$.
\end{enumerate}
\end{cor}

\section{$p$-adic Families of Automorphic Forms}\label{new-families}

In this section, we build on the material from the previous sections to construct $p$-adic families of automorphic forms.  As an application of the differential operators from the prior sections, we obtain the following result:
\begin{thm}\label{thm-rmk2}
Suppose there exists an orbit $\co\in {\mathfrak O}$ such that $\cf(\tau)\neq 0,n$ for all $\tau \in\co$ (see Remark \ref{manysimple}). Let $f$ be a $p$-adic
automorphic form of weight $\kappa$, and let $\{\lambda_n\}_{n\in{\mathbb N}}$ a sequence of simple weights that converges $p$-adically, and satisfies the conditions of Theorem \ref{THETACONG}.   Then the automorphic forms $\Theta^{\lambda_n}(f)$ converge to a $p$-adic form in $V^N[\kappa\cdot\chi]$, for $\chi:=\lim_n \lambda_n$. 
\end{thm}
\begin{proof}
 Together, Theorem \ref{THETACONG} and Corollary \ref{compare} imply the automorphic forms $\Theta^{\lambda_n}(f)$ converge to a $p$-adic form in $V^N[\kappa\cdot\chi]$, for $\chi:=\lim_n \lambda_n$. 
\end{proof}

For applications to $p$-adic $L$-functions and Iwasawa theory, it is often convenient to construct $p$-adic measures.  Recall (e.g., from \cite[Section 4.0]{kaCM}) that for $R$ a $p$-adic ring, an $R$-valued $p$-adic measure on a compact, totally disconnected topological space $Y$ is a $\ZZ_p$-linear map $\mu$ from the $\ZZ_p$-algebra $\CC(Y, \ZZ_p)$ of $\ZZ_p$-valued continuous functions on $Y$ to $R$.  It is equivalent to give an $R$-linear map from the $R$-algebra $\CC(Y, R)$ of $R$-valued continuous functions on $Y$ to $R$, since $\CC\left(Y, \ZZ_p\right)\hat\otimes_{\ZZ_p}R\cong \CC\left(Y, R\right).$  Given $\chi\in\CC(Y, R)$, we write $\int_Y\chi d\mu := \mu(\chi)$.

Let 
\begin{align*}
W=\prod_{\left\{\co | \cf_\co(\tau)\neq {0,n}\forall \tau\in\co\right\}}\ZZ_p^\times.
\end{align*}
So the rank of $W$ is the number of components at which a {\it simple} weight (in the sense of Definition \ref{simple}) can be nonzero.

\begin{thm}\label{thm-rmk1}
Let $f$ be a $p$-adic automorphic form OMOIT.
Then there is a $\Vmu$-valued $p$-adic measure $\mu_f$ on $W$ such that
\begin{align*}
\int_{W} \lambda d\mu_f =\Theta^\lambda(f)
\end{align*}
for all simple positive integer weights $\lambda$.

In particular, if $f$ is of weight $\kappa$ and $\{\lambda_n\}_{n\in{\mathbb N}}$ is a sequence of positive weights that converges $p$-adically and satisfies the conditions of Theorem \ref{THETACONG}, then the automorphic forms $\Theta^{\lambda_n}(f)$ converge to a $p$-adic form in $\Vmu[\chi\cdot\kappa]$, for $\chi:=\lim_n \lambda_n$.
\end{thm}
\begin{proof}
Theorem \ref{THETACONG} and Corollary \ref{Dscalar} combined imply that  the automorphic forms $\Theta^{\lambda_n}(f)$ converge to a $p$-adic form in $\Vmu[\chi]$, for $\chi:=\lim_n \lambda_n$.   The rest of the statement follows from the definition of a $p$-adic measure.
\end{proof}

For applications to $p$-adic $L$-functions, one often needs to relate certain $p$-adic and $\IC$-valued automorphic forms.  For example, Katz's construction of $p$-adic $L$-functions for CM fields in \cite{kaCM} includes a comparison of values of $p$-adic and $\IC$-valued Hilbert modular forms at certain ordinary Hilbert--Blumenthal abelian varieties, in two distinct ways:
\begin{enumerate}
\item{Equate (modulo periods) the values (at ordinary CM Hilbert--Blumenthal abelian varieties defined over $\Dring$) of a $p$-adic and a $\ci$ Hilbert modular form obtained by applying a $p$-adic differential operator (related to the ones in this paper) and the analogous $\ci$ Maass--Shimura operator to a holomorphic Hilbert modular form defined over $\Dring$.}\label{comparison-kaCM}
\item{Express a $p$-adic automorphic form obtained by applying a $p$-adic differential operator (analogous to the ones in this paper) to an Eisenstein series defined over $\Dring$ as a $p$-adic limit of finite sums of holomorphic (algebraic) Eisenstein series over $\Dring$.}\label{limit-kaCM}
\end{enumerate}
Item \eqref{comparison-kaCM}, in particular, plays a key role in Katz's construction of $p$-adic $L$-functions for CM fields \cite{kaCM}.

\begin{rmk}\label{cor-CMcomparison}
Let $f$ be an algebraic automorphic form arising over $\Dring$.  Then by extension of scalars, we may view $f$ as an automorphic form over $\IC$ or as an automorphic form over $\CO_{\bar{\IQ}_p}$.  Let $\uA$ be a $\mu$-ordinary CM point over $\Dring$, together with a choice of differentials $\omega$ over $\Dring$, and let $c$ be such that $c\omega$ is the canonical basis over the $\mu$-ordinary Igusa tower.  So if $f$ is of weight $\kappa$, $f(\uA, c\omega)=\Omega_{c, \kappa} f(\uA, \omega)$, for some $\Omega_{c, \kappa}$ dependent on $c$ and $\kappa$.

We expect that a similar argument to the one in \cite[Section 5]{kaCM} yields an analogous comparison to \eqref{comparison-kaCM} at $\mu$-ordinary CM points over $\Dring$ (together with a basis of differentials) of $\pi^{\kappa+\lambda} D_\kappa^\lambda(f)$ and $\pi^{\kappa+\lambda}D_{\kappa, \ci}^\lambda(f)$.  
As an illustration of a consequence of this comparison, we provide Corollary \ref{coro-cons}, which interpolates values (modulo periods) at CM points of $\ci$ automorphic forms.
\end{rmk}

\begin{cor}[Corollary to Theorem \ref{thm-rmk1}]\label{coro-cons}
Let $f$ be a weight $\kappa$ algebraic automorphic form arising over $\Dring$, and let $(\uA, \omega)$, $c$, and $\Omega_{c, \kappa}$ be as in Remark \ref{cor-CMcomparison}.  Then
\begin{align*}
\frac{1}{\Omega_{c, \kappa+\lambda}}\int_{W} \lambda d\mu_f =\left(\pi^{\kappa+\lambda} D_{\kappa, \ci}^\lambda(f)\right)(\uA, \omega)
\end{align*}
\end{cor}
\begin{proof}
This is a consequence of Theorem \ref{thm-rmk1}, combined with Remark \ref{cor-CMcomparison}.
\end{proof}

\begin{rmk}
As noted in Remark \ref{goodies}, when $p$ splits completely in $\reflexfield$, the ordinary locus is always nonempty.  Thus, \cite[Theorem 7.2.3]{EFMV}, which obtains an explicit family of automorphic forms by applying differential operators to a family of Eisenstein series on a unitary group $G$ of signature $(n,n)$ and then restricting to a subgroup $G'$ of $G$, can be extended to the case where $p$ need not split completely in $\cmfield$ but merely splits completely in $\reflexfield$ (replacing the stronger condition that $p$ splits completely in $\cmfield$).  The approach in the proof of \cite[Theorem 7.2.3]{EFMV} uses the existence of an ordinary cusp for $G$ together with the inclusion of the $\mu$-ordinary locus for $G'$ inside the ordinary locus for $G$, which only exists when the ordinary locus for $G'$ is nonempty.  Ideally, we would also like to handle the case where the ordinary locus for $G'$ is empty.  We expect that the analogue in our setting of Hida's ordinary projection ({\it $\mu$-ordinary} or {\it $P$-ordinary} projection, in our case) will help enable such as an extension.
\end{rmk}

\bibliography{muordinarybibAJMstyle}
\end{document}